\documentclass[a4paper, 10pt, twoside, notitlepage]{amsart}

\usepackage{geometry}
\usepackage[utf8]{inputenc}
\usepackage{color}
\usepackage{amsmath}
\usepackage{amssymb}
\usepackage{amsthm}
\usepackage{graphicx}
\usepackage{esint}
\usepackage[colorlinks=true,linkcolor=blue]{hyperref}
\usepackage{verbatim}
\usepackage[mathscr]{euscript}
\newtheorem{thm}{Theorem}[section]
\newtheorem{cor}{Corollary}

\newtheorem{lemma}[thm]{Lemma}
\newtheorem{prop}{Proposition}

\theoremstyle{definition}
\newtheorem{rmk}{Remark}

\newcommand {\R} {\mathbb{R}} 
 \newcommand {\N} {\mathbb{N}}

\newcommand {\p} {\partial}

\DeclareMathOperator{\di}{div}

\DeclareMathOperator {\dist} {dist}
\DeclareMathOperator {\Ree} {Re}

\DeclareMathOperator{\inte} {int}

\usepackage{xspace}

\begin{document}

\title[Epiperimetric inequality approach to Parabolic Signorini problem]{An epiperimetric inequality approach to the parabolic Signorini problem}

\author{Wenhui Shi}

\address{9 Rainforest Walk, Level 4, Monash University, VIC 3800, Australia}
\email{wenhui.shi@monash.edu}

\begin{abstract}
In this note, we use an epiperimetric inequality approach to study the regularity of the free boundary for the parabolic Signorini problem. We show that if the "vanishing order" of a solution at a free boundary point is close to $3/2$ or an even integer, then the solution is asymptotically homogeneous. Furthermore, one can derive a convergence rate estimate towards the asymptotic homogeneous solution. As a consequence, we obtain the regularity of the regular free boundary as well as the frequency gap.
\end{abstract}



\maketitle

\tableofcontents

\section{Introduction}
In this note, we develop a new approach to study the asymptotics of solutions to the parabolic Signorini problem at the free boundary points. More precisely, let $u$ be a solution to
\begin{equation}\label{eq:parabolic}
\begin{split}
\p_t u-\Delta u=f &\text{ in } S_2^+\\
u\geq 0, \ \p_n u\leq 0, \ u\p_n u=0 &\text{ on } S'_2 .
\end{split}
\end{equation}
Here for $R>0$ and space dimension $n\geq 2$
\begin{equation*}
\begin{split}
S_R^+&:=\{(x,t): x\in \R^n_+, \ t\in [-R,0]\},\quad \R^n_+:=\R_n\cap \{x_n>0\}\\
S'_R &:=\{(x,t): x\in \R^n\cap \{x_n=0\}, \ t\in [-R,0]\}
\end{split}
\end{equation*}
and $f\in L^\infty(S_2^+)$ is a given inhomogeneity. Throughout the paper we will assume that
\begin{itemize}
\item [(A)] $u$ is normalized such that
\begin{align*}
\int_{S_2^+} u(x,t)^2 G(x,t)dxdt=1,
\end{align*}
where
\begin{align*}
G(x,t):=\left\{
          \begin{array}{ll}
            (-4\pi t)^{-\frac{n}{2}}e^{\frac{|x|^2}{4t}}, & \ t<0, \\
            0, & \ t\geq 0
          \end{array}
        \right.
\end{align*}
denotes the backward heat kernel in $\R^n\times \R$.
\item[(B)] $u$ satisfies the Sobolev regularity in the Gaussian space: there exists $C$ depending on $n$ and $\|f\|_{L^\infty(S_2^+)}$ such that
\begin{align*}
&\int_{S_1^+}(-t)^2\left(|D^2u(x,t)|^2+|\p_tu(x,t)|^2\right)G(x,t)dxdt\\
&+\sup_{t\in [-1,0]} \int_{\R^n_+} (|\nabla u(x,t)|^2 +|u(x,t)|^2) G(x,t) dx \leq C.
\end{align*}
\item [(C)] $u$ satisfies the interior H\"older estimate: there exists $\alpha\in (0,1)$  such that for any $U\Subset \R^n_+\cup (\R^{n-1}\times\{0\})$,
\begin{equation*}
\|\nabla u\|_{C^{\alpha,\alpha/2}(U\times [-1,0])}\leq C
\end{equation*}
for some $C>0$ depending on $n$, $\alpha$,  $\|f\|_{L^\infty(S_2^+)}$ and $U$. Here $C^{\alpha,\alpha/2}$ is the parabolic H\"older class.
\end{itemize}

Solutions to \eqref{eq:parabolic} can come from solving a variational inequality for the initial value problem in the class of functions with mild growth at infinity and satisfy $u\in L^2([-2,0]; W^{1,2}_{loc}(\R^n_+))$, $\p_t u\in L^2([-2,0];L^2_{loc}(\R^n_+))$, or they can come from solutions to the Signorini problem in a bounded domain (one applies suitable cut-offs to extend them into full space solutions). In both cases, the Sobolev estimate in (B) and the interior H\"older estimate in (C) hold true, cf. \cite{AU, DGPT}. The normalization assumption (A) is put simply to make the constants in (B) and (C) independent of $u$. Under our regularity assumption, the Signorini boundary condition in \eqref{eq:parabolic} holds in the classical sense.

We denote the contact set
\begin{equation*}
\Lambda_u:=\{(x,t)\in S'_2: u(x,t)=0\}
\end{equation*}
and the free boundary
\begin{equation*}
\Gamma_u:=\p_{S'_2}\Lambda_u.
\end{equation*}

\subsection{Main results} The behavior of a solution around a free boundary point depends very much on how fast it vanishes towards it. The main results of the paper are the following: we prove that if the "vanishing order" of a solution at a free boundary point is close to $3/2$, which is the expected lowest vanishing order, or $2m$, $m\in \N_+$, which are the eigenvalues of the Ornstein-Uhlenbeck operator $-\frac{1}{2}\Delta-x\cdot \nabla$ in $\R^n_+$ with the vanishing Neumann boundary condition, then the solution is asymptotically a homogeneous solution.  As a consequence, we derive the frequency gap around the free boundary points with vanishing order $3/2$ and $2m$; the openness of the regular free boundary (consisting the free boundary points with $3/2$ vanishing order), as well as its space-time regularity.

More precisely, assume $(0,0)\in \Gamma_u$, and for $r\in (0,1)$, let
$$H_u(r):=\frac{1}{r^2}\int\limits_{S_{r^2}^+}u(x,t)^2 G(x,t)dxdt$$
be the weighted $L^2$ space-time average at $(0,0)$. The first theorem is about the asymptotics of the solution around the free boundary point where the vanishing order is around $3/2$:

\begin{thm}\label{thm:32}
Let $u$ be a solution to \eqref{eq:parabolic} and satisfy the assumptions (A)-(C). Assume that $(0,0)\in \Gamma_u$. Then there exists a constant $\gamma_0\in (0,1)$ depending on $n$ and $\|f\|_{L^\infty}$ such that if
\begin{equation}\label{eq:vanishing_32}
c_u r^{3+\gamma_0}\leq H_u(r)\leq C_u r^{3-\gamma_0}, \text{ for all } r\in (0,r_u)
\end{equation}
for some $r_u, c_u, C_u>0$, then there exists a unique function $u_0(x)=c_0 \Ree(x'\cdot e_0+ i |x_n|)^{3/2}$ with $c_0>0$ and $e_0\in \mathbb{S}^{n-1}\cap \{x_n=0\}$ such that
\begin{equation*}
\int\limits_{\R^n_+} |u(x,t)-u_0(x)|^2 G(x,t)dx \leq C (\sqrt{-t})^{3 + 2\gamma_0}
\end{equation*}
for all $t\in [-1,0)$. Here $C>0$ depends on $n$ and $\|f\|_{L^\infty(S_2^+)}$.
\end{thm}

Theorem \ref{thm:32} still holds true, if we assume instead of the lower bound in \eqref{eq:vanishing_32} that the solution at $t=-1$ is sufficiently close to the asymptotic profile, cf. Remark \ref{rmk:non_zero2}. We remark that \eqref{eq:vanishing_32} is satisfied if $(0,0)$ is a free boundary point with frequency $\kappa=3/2$, where the frequency is defined as the limit of the Almgren-Poon frequency function at $(0,0)$, cf. \cite{DGPT} for the precise definition. However, our assumption \eqref{eq:vanishing_32} is much weaker, since it does not rely on the existence of the limit of the frequency function or the optimal spacial regularity. On the other hand, if $\alpha\geq \frac{1-\gamma_0}{2}$ in the assumption (C), then we can instead derive the optimal interior regularity $\nabla u\in C^{1/2,1/4}$ from Theorem \ref{thm:32}.

The next theorem is about the logarithmic convergence towards the asymptotic homogeneous solution around the free boundary point with the vanishing order close to $2m$. For that, we let
$\mathcal{E}_{2m}$ be the $2m$-eigenspace of the Ornstein-Uhlenbeck operator $-\frac{1}{2}\Delta + x\cdot \nabla $ on $\R^n_+$ with the vanishing Neumann boundary condition on $\{x_n=0\}$.  Let $\mathcal{E}_{2m}^+$ be the convex cone in $\mathcal{E}_{2m}$ where the restriction of $p$ on $\{x_n=0\}$ is nonnegative. We remark that given $\bar p\in \mathcal{E}_{2m}^+$, the function $p(x,t)=(\sqrt{-t})^{2m}\bar p(\frac{x}{2\sqrt{-t}})$, $t\in (-\infty,0)$, is a $2m$-parabolic homogeneous solution to \eqref{eq:parabolic}.

\begin{thm}\label{thm:2m}
Let $u$ be a solution to \eqref{eq:parabolic} and satisfy the assumptions (A)-(C) with $f=0$. Let $m\in \N_+$ be an arbitrary positive integer. Assume that $(0,0)\in \Gamma_u$. Then there exist small constants $\gamma_m, \delta_0\in (0,1)$ depending only on $m, n$, such that if
\begin{equation*}
H_u(r)\leq C_u r^{2\kappa-\gamma_m}, \quad \kappa=2m,
\end{equation*}
for each $r\in (0,r_u)$ for some $r_u,C_u>0$, and at $t=-1$
\begin{equation*}
\dist_{W^{1,2}_{\tilde{\mu}}(\R^n_+)}(u(\cdot,-1), \mathcal{E}_{2m}^+)^2\leq \delta_0\|u(\cdot,-1)\|_{L^2_{\tilde{\mu}}(\R^n_+)}^2, \quad d\tilde{\mu}(x):=G(x,-1)dx=c_ne^{-\frac{|x|^2}{4}}dx,
\end{equation*}
then there exists a unique nonzero $p_0(x,t)=(\sqrt{-t})^{2m}\bar p_0(\frac{x}{2\sqrt{-t}})$ with $\bar p_0\in \mathcal{E}_{2m}^+$ such that
\begin{equation*}
\begin{split}
\int\limits_{\R^n_+}|u(x,t)- p_0(x,t)|^2 G(x,t) dx&\rightarrow 0\text{ as } t\rightarrow 0,\\ 
\int_{\R^n_+}\left||u(x,t)|^2-|p_0(x,t)|^2\right|G(x,t) dx&\leq \frac{C(\sqrt{-t})^{2m}}{|\ln\ln(-t)|} \text{ for all }t\in [-1,0).
\end{split}
\end{equation*}
\end{thm}

It is possible to generalize Theorem \ref{thm:2m} to a nonzero inhomogeneity $f$, where $f$ satisfies an additional vanishing property at $(0,0)$: $|f(x,t)|\leq M(|x|+\sqrt{-t})^{2(m-1+\epsilon_0)}$ for some $\epsilon_0>0$ and $M>0$, cf. Section \ref{sec:perturbation}. We state and prove the theorem for $f=0$ to avoid the technicalities caused by the inhomogeneity such that the proofs are neater. The assumptions of Theorem \eqref{thm:2m} are satisfied at the free boundary points with the frequency $2m$, cf. \cite{DGPT}, but here we do not rely on the existence of the limit of the frequency function.

Note that in Theorem \ref{thm:2m} we obtain a logarithmic decay (in $\ln(-t)$) of the $L^2$ norm instead of an exponential decay as in Theorem \ref{thm:32}. A polynomial decay rate of this kind towards the asymptotic solutions was obtained originally in the elliptic case \cite{CSV17,CSV18, CSV18II}. Moreover, in the classical obstacle problem it was shown, that there is in general no exponential decay rate at singular points, cf. \cite{FS18}.

\subsection{Main ideas of the proof} The proofs for Theorem \ref{thm:32} and Theorem \ref{thm:2m} are based on a \emph{dynamical system approach}, where we establish a decay rate for the Weiss energy $W_{\kappa}$, $\kappa=3/2$ and $\kappa=2m$, in the self-similar conformal coordinates. In the elliptic problems, such change of coordinates corresponds to $(r,\theta)\mapsto (t,\theta)=(-\ln r, \theta)$ from $(0,1]\times \mathbb{S}^{n-1}$ to $[0,\infty)\times \mathbb{S}^{n-1}$, which transforms the original problem around the free boundary point at $0$ to a dynamical system on $\mathbb{S}^{n-1}$. The equilibrium of the dynamical system then corresponds to (back to the original coordinates) the blow-up limits at the origin. In the parabolic setting, we can (formally) formulate our problem in the self-similar conformal coordinates as a gradient flow of the Weiss energy under the convex constraint $u\geq 0$ on $\{x_n=0\}$. The relation between the Weiss energy and the evolution of certain quantities thus becomes more transparent, cf. Lemma \ref{lem:Weiss_conf}. Very different from the elliptic problem, in the parabolic setting the dynamical system is on the whole space $\R^n$ (instead of $\mathbb{S}^{n-1}$). The non-compactness brings additional difficulties and we could not get the polynomial decay rate towards general blow-ups at singular points as in the elliptic setting (cf. \cite{CSV17, CSV18}). It is unlikely that the above decay rate is optimal, thus it is a very interesting open problem to improve the decay rate and further explore the optimality. 

The main steps in the proof are discrete decay estimates for the Weiss energy $W_{\kappa}$, $\kappa=3/2$ and $\kappa=2m$, which can be viewed as parabolic epiperimetric inequalities. Epiperimetric inequalities were introduced by G. Weiss \cite{Weiss99} to the classical obstacle problem and they continue to be a subject of intense research interest in the elliptic setting \cite{CSV17,CSV18,CSV18II, FS16, GPSVG15,RS17}. 

\subsection{Known results} Existence and uniqueness of solutions for given initial and boundary data follow from the classical theory of variational inequalities, cf. for instance, \cite{DL76}. For sufficiently regular inhomogeneities, solutions $u$ satisfy the pointwise regularity $\nabla_x u\in H^{1/2,1/4}_{loc}$ (H\"older $1/2$ in space and $1/4$ in time) and $\p_t u\in L^\infty_{loc}$, cf. \cite{AU, ACM18, DGPT, PZ19}. Such regularity is optimal at least in the space variables due to the stationary solution $\Ree(x_{n-1}+i|x_n|)^{3/2}$. Classification of free boundary points is based on the Almgren and Poon's type monotonicity formulas, cf. \cite{DGPT}. The regular free boundary $\Gamma_{3/2}$ consists of those free boundary points which have the lowest frequency $3/2$. It is relatively open (possibly empty) and locally a smooth graph given by $x_{n-1}=f(x'',t)$ up to a rotation for smooth inhomogeneities, cf. \cite{ASZ17,PS14,PZ19}. The so-called singular set consists of those free boundary points whose frequency are even integers $2m$, $m\geq 1$. Singular set can be equivalently characterized by the density of the contact set, i.e. it consists of free boundary points at which the contact set $\Lambda_u$ has zero Lebesgue density, cf. \cite{DGPT}. Finer structure of the singular set is studied in \cite{DGPT}. 

We briefly comment about our results and the related literature. Theorem \ref{thm:32} allows to provide a simpler proof for the openness of the regular free boundary and its space-time regularity (cf. Section \ref{sec:consequences}), which was shown in \cite{ACM18, DGPT, PS14} by using a boundary Harnack approach. In particular, we do not require the optimal spacial regularity or continuity of the time derivative of the solutions. Theorem \ref{thm:2m} generalizes the results about the singular set in \cite{DGPT} in the sense that we obtain a log-log modulus of continuity of the $L^2$ norm, which implies a frequency gap around the $2m$-frequency points (cf. Section \ref{sec:consequences}). We mention that compared with the recent progress on the fine structure of the singular set for elliptic problems, for instance \cite{FJ18,FS18,FSpadaro18, FS19, STT18}, the parabolic counterpart is much less known. The logarithmic epiperimetric inequality was recently established for the elliptic obstacle and thin obstacle problems \cite{CSV17,CSV18,CSV18II}, which inspired our paper. Very different from the proofs in \cite{CSV17, CSV18, GPSVG15} and the existing epiperimetric inequality approach for the parabolic problem (where they reduce the parabolic problem to a stationary problem by treating the time derivative as inhomogeneity), our approach is purely dynamical and does not rely on the (almost) minimality. Hence it can possibly be generalized to study the asymptotics around \emph{critical points} instead of energy minimizers for a broader class of free boundary problems.  
It is also possible to apply our approach to the Signorini problem for the degenerate parabolic operators considered in \cite{ACM18,AT18,BDGP,BDGPII,BG18}.\\

The rest of the paper is structured as follows: in Section \ref{sec:conformal} we introduce the conformal change of coordinates, reformulate our problem in the new coordinates, and explore the role of the Weiss energy; In Section \ref{sec:dynamical} we prove discrete
decay estimates for the Weiss energy $W_\kappa$. For simplicity we assume that the inhomogeneity $f$ vanishes. The idea of the proof remains the same for nonzero inhomogeneities, as one can see from Section \ref{sec:perturbation}. In Section \ref{sec:consequences} we show the consequences of the decay estimates, for example, we prove the $C^{1,\alpha}$ regularity of the regular free boundary and the frequency gap around $2m$-frequency. In the last section we will show how to modify the proof in Section \ref{sec:dynamical} to the inhomogeneous setting. In the appendix we provide a short proof for the characterization of the second Dirichlet-Neumann eigenspace for the Ornstein-Uhlenbeck operator.

\section{Conformal self-similar coordinates and Weiss energy}
\label{sec:conformal}

In the sequel, it will prove convenient to work in conformal self-similar coordinates. This will simplify many of the computations, which will be carried out for the Weiss energy.

Thus, we consider the following change of variables, which should be viewed as the analogue of conformal polar coordinates:

\begin{lemma}
\label{lem:coordinates}
Let $u:S_2^+ \rightarrow \R$ be a solution to \eqref{eq:parabolic}. We consider the change of coordinates
\begin{equation}\label{eq:conf}
\begin{split}
\mathcal{T}:\R^{n}\times [-1,0)&\rightarrow \R^n\times[0,\infty)\\
(x,t)&\mapsto (y,\tau)= \left(\frac{x}{2\sqrt{-t}}, -\ln(-t)\right).
\end{split}
\end{equation}
For $\kappa>0$ we denote
\begin{align*}
\tilde{u}_\kappa(y,\tau):=\frac{u(x,t)}{(\sqrt{-t})^\kappa}\big |_{(x,t)=\mathcal{T}^{-1}(y,\tau)}=e^{\tau\kappa/2}u(2e^{-\tau/2}y, -e^{-\tau}).
\end{align*}
Then, $\tilde{u}_{\kappa}$ is a solution to
\begin{align}
\label{eq:bulk_conf}
\p_\tau\tilde{u}+\frac{y}{2}\cdot \nabla \tilde{u}-\frac{1}{4}\Delta \tilde{u}-\frac{\kappa}{2}\tilde{u}=0 \text{ in } \R^n_{+}\times [0,\infty)
\end{align}
with the Signorini condition
\begin{align}
\label{eq:boundary_conf}
\tilde{u}\geq 0, \  \p_n\tilde{u}(y)\leq 0, \ \tilde{u}\p_n\tilde{u}=0 \text{ on } \{y_n=0\}.
\end{align}
\end{lemma}

\begin{proof}
The proof follows from a direct computation.
\end{proof}

\begin{rmk}
\label{rmk:reverse}
For later use, we remark that the above change of coordinates is reversible. In particular, if $\tilde{u}_\kappa$ is a stationary solution to \eqref{eq:bulk_conf} and \eqref{eq:boundary_conf}, then the function
\begin{align*}
u(x,t):= (\sqrt{-t})^{\kappa} \tilde{u}_\kappa\left( \frac{x}{2 \sqrt{-t}} \right), \ t<0
\end{align*}
is a \emph{parabolically $\kappa$-homogeneous} solution to \eqref{eq:parabolic}, i.e. $u(x,t)$ solves the equation \eqref{eq:parabolic} and satisfies $$u(\lambda x, \lambda^2 t)= \lambda^{\kappa} u(x,t)\text{ for all }\lambda >0.$$
\end{rmk}

\begin{rmk}\label{rmk:Gaussian_bound}
Let $u:S_1^+\rightarrow \R$ satisfy the Sobolev regularity (A) and (B). Let $\tilde{u}_\kappa$ be obtained from $u$ as in Lemma \ref{lem:coordinates}. Then it holds
 $$\|\tilde{u}_\kappa(y,0)\|_{L^2_\mu(\R^n_+)}=\pi^{n/2}\|u(x,-1)\|_{L^2_{\tilde{\mu}}(\R^n_+)}, \quad d\mu(y):=e^{-|y|^2}dy, \ d\tilde{\mu}(x):=e^{-\frac{|x|^2}{4}}dx.$$
Moreover, for any $0\leq \tau_1<\tau_2<\infty$, there exists a constant $C$ depending on $\tau_1, \tau_2, \kappa, n, \|f\|_{L^\infty}$ such that
\begin{equation*}
\begin{split}
\int_{\tau_1}^{\tau_2} \|D^2\tilde{u}_\kappa\|^2_{L^2_\mu} + \|\p_\tau\tilde{u}_\kappa\|^2_{L^2_\mu} d\tau \leq C,\\
\sup_{\tau\in (0,\tau_2]}\|\nabla\tilde{u}_\kappa(\cdot, \tau)\|_{L^2_\mu}+ \|\tilde{u}_\kappa(\cdot,\tau)\|_{L^2_\mu}\leq C.
\end{split}
\end{equation*}
Here and in the sequel $\|\cdot\|_{L^2_\mu}:=\|\cdot\|_{L^2_\mu(\R^n_+)}$. Thus
\begin{equation}\label{eq:sob_conf}
\tilde{u}_\kappa \in L^\infty_{loc}([0,\infty); W^{1,2}_\mu)\cap L^2_{loc}([0,+\infty);W^{2,2}_\mu),\ \p_\tau\tilde{u}_\kappa\in L^2_{loc}([0,\infty); L^2_\mu).
\end{equation}
\end{rmk}

Now we define the Weiss energy associated to a solution $\tilde{u}:=\tilde{u}_\kappa$ to \eqref{eq:bulk_conf}--\eqref{eq:boundary_conf} and derive relevant quantities of the Weiss energy.

\begin{lemma}
\label{lem:Weiss_conf}
Let $\tilde{u}=\tilde{u}_\kappa$ be a solution to \eqref{eq:bulk_conf}--\eqref{eq:boundary_conf} and satisfy \eqref{eq:sob_conf}. We define the Weiss energy
\begin{equation}
\label{eq:Weiss_new}
\begin{split}
W_\kappa(\tilde{u}(\tau)):=\int_{\R^n_+}\frac{1}{4}|\nabla \tilde{u}(y,\tau)|^2-\frac{\kappa}{2}\tilde{u}^2(y,\tau) d\mu.
\end{split}
\end{equation}
Then for $0\leq \tau_1 \leq \tau_2 < \infty $
\begin{align}\label{eq:Weiss_new2}
W_{\kappa}(\tilde{u}(\tau_2)) - W_{\kappa}(\tilde{u}(\tau_1))
\leq -2 \int\limits_{\tau_1}^{\tau_2} (\p_{\tau} \tilde{u}(y,\tau))^2 d\mu d\tau.
\end{align}
Further,
\begin{align}\label{eq:w2}
W_\kappa(\tilde{u}(\tau))=-\frac{1}{2}\p_\tau \|\tilde{u}(\tau)\|_{L^2_\mu}^2, \ a.e.\ \tau\in [0,\infty).
\end{align}
\end{lemma}

\begin{proof}
On a formal level the estimate \eqref{eq:Weiss_new2} can be deduced by differentiating the functional $W_{\kappa}(\tilde{u}(\tau))$. However, in order to give meaning to the arising boundary contributions, it is necessary to work in a regularized framework which is achieved by penalization. More precisely, we consider the following penalized version of \eqref{eq:boundary_conf}, \eqref{eq:bulk_conf}: Let $\beta_{\epsilon}: \R \rightarrow \R$ be a smooth function satisfying the following properties
\begin{align*}
\beta_{\epsilon}(s) &=0 \mbox{ for } s \geq 0,\\
\beta_{\epsilon}(s) &= \epsilon + \frac{s}{\epsilon} \mbox{ for all } s \leq - 2 \epsilon^2,\\
\beta_{\epsilon}'(s) &\geq 0 \mbox{ for all } s \in \R.
\end{align*}
We approximate $\tilde{u}$ by solutions to the penalization problem
\begin{align*}
\p_{\tau} \tilde{u}^{\epsilon} + \frac{y}{2}\cdot \nabla \tilde{u}^{\epsilon} - \frac{1}{4} \Delta \tilde{u}^{\epsilon} - \frac{\kappa}{2} \tilde{u}^{\epsilon} & = 0 \mbox{ on } \R^n_+ \times (0,\tau_2),\\
\p_{n} \tilde{u}^{\epsilon} & = \beta_{\epsilon}(\tilde{u}^{\epsilon}) \mbox{ on } \{y_n=0\} \times (0,\tau_2) ,\\
\tilde{u}^{\epsilon}(y,0)&= \tilde{u}_0^{\epsilon}(y) \mbox{ at } \R^n_+ \times \{0\}.
\end{align*}
Here $\tilde{u}_0^{\epsilon}$ is a smooth compact supported function such that $\|\tilde{u}_0^\epsilon-\tilde{u}(\cdot, 0)\|_{L^2_\mu}\rightarrow 0$ as $\epsilon\rightarrow 0$.  There exists a unique solution $\tilde{u}^\epsilon$ with a polynomial growth as $|y|\rightarrow \infty$. The function $\tilde{u}^{\epsilon}$ is smooth and satisfies the uniform bound in the Gaussian space (cf. Chap. 3 in \cite{DGPT}): there exists $C=C(\kappa, \tau_2, \|\tilde{u}^\epsilon_0\|_{L^2_\mu}, n)$ such that
\begin{equation*}
\|D^2\tilde{u}^\epsilon\|_{L^2_\tau L^2_\mu}+\|\p_\tau\tilde{u}^\epsilon\|_{L^2_\tau L^2_\mu}+ \|\tilde{u}^\epsilon\|_{L^\infty_\tau L^2_\mu}+\|\nabla\tilde{u}^\epsilon\|_{L^\infty_\tau L^2_\mu}\leq C.
\end{equation*}
In particular, using the equation for $\tilde{u}^{\epsilon}$, it is then possible to compute as follows:
\begin{align*}
\frac{d}{d\tau} W_\kappa(\tilde{u}^{\epsilon}(\tau))
= - 2 \int\limits_{\R^n_+} (\p_{\tau} \tilde{u}^{\epsilon})^2 d\mu  - \frac{1}{2}\int\limits_{\{y_n=0\}} \beta_\epsilon(\tilde{u}^\epsilon)(\p_{\tau} \tilde{u}^{\epsilon}) d\mu.
\end{align*}
We test this identity with $\varphi \in C^{\infty}_c((0,\infty))$ which yields
\begin{align*}
\int\limits_{0}^{\infty} \varphi(\tau) \frac{d}{d\tau} W_\kappa(\tilde{u}^{\epsilon}(\tau)) d\tau
&= -2 \int\limits_{0}^{\infty} \varphi(\tau) \int\limits_{\R^n_+} (\p_{\tau} \tilde{u}^{\epsilon})^2 d\mu d\tau\\
& \quad - \frac{1}{2} \int\limits_{0}^{\infty } \varphi(\tau) \int\limits_{\{y_n=0\}}  \beta_\epsilon(\tilde{u}^\epsilon)(\p_{\tau} \tilde{u}^{\epsilon}) d\mu d \tau.
\end{align*}
The a priori bounds for $\tilde{u}^{\epsilon}$ leads to space time $W^{1,2}_{\mu}$ uniform estimates. Hence it is possible to use lower semi-continuity to pass to the weak limit in the bulk integral. Using that $\beta_\epsilon(\tilde{u}^\epsilon)\p_{\tau} \tilde{u}^{\epsilon} =\p_\tau \mathcal{B}_\epsilon (\tilde{u}^\epsilon)$, where $\mathcal{B}_\epsilon$ is the primitive function of $\beta_\epsilon$ with $\mathcal{B}_\epsilon(0)=0$, the boundary integral is treated as in Lemma 5.1 ($3^\circ$) in \cite{DGPT} and can be shown to vanish in the limit. Hence for $\varphi\geq 0$ we infer
\begin{align*}
-\int\limits_{0}^{\infty} \varphi'(\tau) W_{\kappa}(\tilde{u}(\tau)) d\tau
\leq  -2 \int\limits_{0}^{\infty} \varphi(\tau) \int\limits_{\R^n_+} (\p_{\tau} \tilde{u})^2 d\mu d\tau,
\end{align*}
which yields the desired result. Approximating the characteristic function $\chi_{[\tau_1,\tau_2]}(t)$ by smooth positive functions then yields the claim on the sign of the difference of the Weiss functionals.

Finally, the identity \eqref{eq:w2} is a consequence of the equation \eqref{eq:boundary_conf} in conjunction with the Signorini condition $\tilde{u}\p_n\tilde{u}=0$ on $\{y_n=0\}$.
\end{proof}

As direct consequences of Lemma \ref{lem:Weiss_conf}, we infer the following properties:

\begin{cor}
\label{cor:Weiss}
Let $\tilde{u}:\R^n_+ \times [0,\infty) \rightarrow \R $ be a solution to \eqref{eq:boundary_conf}, \eqref{eq:bulk_conf} which satisfies \eqref{eq:sob_conf}. Then,
the function
\begin{align*}
\tau\mapsto \|\tilde{u}(\tau)\|_{L^2_\mu}^2
\end{align*}
is convex.
Moreover, $\tilde{u}$ is a stationary solution if and only if $W_{\kappa}(\tilde{u}(\tau))\equiv 0$.
\end{cor}

In the sequel, we will in particular exploit the second observation frequently. When $\kappa=3/2$ and $\kappa=2m$, $m\in \N_+$, by Liouville type theorems, we can characterize stationary solutions.

\begin{prop}\label{prop:wlimit}
Let $\tilde{u}\in W^{2,2}_\mu$ be a solution to
\begin{equation}\label{eq:static}
\begin{split}
-\frac{1}{2}\Delta \tilde{u} + y\cdot \nabla \tilde{u} -\kappa \tilde{u}&=0 \text{ in } \R^n_+,\\
\tilde{u}\geq 0, \ \p_n\tilde{u}\leq 0, \ \tilde{u}\p_n\tilde{u}&=0 \text{ on } \R^{n-1}\times \{0\}.
\end{split}
\end{equation}
We extend $\tilde{u}$ evenly about $\{y_n=0\}$. If $\kappa=3/2$, then
\begin{equation*}
\tilde{u}\in \mathcal{E}_{3/2}:=\{c\Ree (y'\cdot e + i|y_n|)^{3/2}, \ c\geq 0,\ |e|=1, \ e\cdot e_n=0.\}
\end{equation*}
If $\kappa=2m$ for $m\in \N_+$, then
\begin{equation*}
\begin{split}
\tilde{u}\in \mathcal{E}_{2m}^+:=\{p: p=\sum_{\alpha:|\alpha|=2m}\lambda_\alpha H_{\alpha_1}(y_1)\cdots H_{\alpha_n}(y_n), \ \lambda_\alpha\in \R, \\
 p(y',y_n)=p(y',-y_n), \ p(y',0)\geq 0\}.
\end{split}
\end{equation*}
Here $\alpha=(\alpha_1,\cdots, \alpha_n)\in \N^n$ is a multi-index, $H_{\alpha_i}$ is the $1d$-Hermite polynomial of order $\alpha_i$, $i\in \{1,\cdots, n\}$.
\end{prop}
\begin{proof}
\emph{(i) Case $\kappa=3/2$.} We will prove that  $\tilde{u}$ is two dimensional. Given any tangential direction $e$ with $|e|=1$ and $e\cdot e_n=0$, $v:=\p_e \tilde{u}$ solves the Dirichlet eigenvalue problem for $L_0:=-\frac{1}{2}\Delta + y\cdot \nabla$ on $W^{1,2}_0(\R^n\setminus \Lambda_{\tilde{u}}; d\mu)\subset L^2_\mu(\R^n)$:
\begin{equation*}
L_0 v = \frac{1}{2}v \text{ in } \R^n\setminus \Lambda_{\tilde{u}}, \quad v=0\text{ on }\Lambda_{\tilde{u}},
\end{equation*}
where $\Lambda_{\tilde{u}}:=\{(y',0): \tilde{u}(y',0)=0\}$. We claim that $v$ does not change the sign in $\R^n$. Let $0<\lambda_1\leq \lambda_2\leq \cdots$ denote the Dirichlet eigenvalues. Assume that $v$ changes the sign, then necessarily $\lambda_2\leq \frac{1}{2}$. By the min-max theorem,
\begin{equation*}
\begin{split}
 \lambda_2=&\inf\left\{\sup\left\{\frac{1}{2}\int |\nabla v|^2 d\mu: v\in M, \ \|v\|_{L^2_\mu}=1\right\}:\right.\\
  &\qquad \left.M\subset W_0^{1,2}(\R^n\setminus \Lambda_{\tilde{u}};d\mu) \text{ subspace and } dim(M)=2\right\}.
\end{split}
\end{equation*}
We consider the $2d$ subspaces $M$ spanned by the (oddly reflected) Dirichlet eigenfunctions on $\R^n_+$ with $w=0$ on the whole $\R^{n-1}\times \{0\}$. Since $w(y)=y_n$ is the first Dirichlet eigenfunction, whose Rayleigh quotient is equal to $1$, then necessarily $\lambda_2\geq 1$. This is a contradiction. Therefore we conclude that $v=\p_e\tilde{u}$ is nonpositive or nonnegative in the whole space $\R^n$. Since this holds for any tangential direction, it follows that $\tilde{u}$ is of the form $\tilde{u}(y)=\tilde{u}(y'\cdot e, y_n)$ for some tangential direction $e$. In other words, $\tilde{u}$ is two dimensional. Direct computation shows that the function $\Ree(y'\cdot e+ i|y_n|)^{1/2}$ is an eigenfunction. The uniqueness of the principal eigenfunction implies that actually $\p_e \tilde{u}= c\Ree(y'\cdot e+ i|y_n|)^{1/2}$ for some $c\in \R$. Thus $\tilde{u}\in \mathcal{E}_{3/2}$.

\emph{(ii) Case $\kappa=2m$.} This follows directly from Lemma 12.4 of \cite{DGPT} and Remark \ref{rmk:reverse}.
\end{proof}

At the end of this section we compare the Weiss energy in the original coordinates and the conformal coordinates. Firstly, if $\tilde{u}_\kappa$ is associated with a solution $u:S_2^+ \rightarrow \R$ to the parabolic Signorini problem
\eqref{eq:parabolic} as in Lemma \ref{lem:coordinates}, then the Weiss energy of $\tilde{u}_\kappa$ in \eqref{eq:Weiss_new} can be rewritten in terms of $u$ as
\begin{align}\label{eq:weiss_original}
W_\kappa(u(t))=\frac{1}{(-t)^{\kappa-1}} \int\limits_{\R^n_+} |\nabla u(x,t)|^2 G(x,t)dx-\frac{\kappa}{2(-t)^\kappa}\int\limits_{\R^n_+}u(x,t)^2G(x,t)dx.
\end{align}
Next, for $\lambda>0$, let
\begin{align*}
u_\lambda(x,t):=\frac{u(\lambda x, \lambda^2 t)}{\lambda^\kappa},
\end{align*}
be the (parabolic) $\kappa$-homogeneous scaling, and let
\begin{align*}
 \tilde{u}_{\kappa,\lambda}(y,\tau):=\frac{u_\lambda(x,t)}{(\sqrt{-t})^\kappa}\big|_{(x,t)=\mathcal{T}^{-1}(y,\tau)}
\end{align*}
as in Lemma \ref{lem:coordinates}. Then
\begin{align}
\label{eq:kappa}
\tilde{u}_{\kappa,\lambda}(y,\tau)=\tilde{u}_\kappa(y,\tau-2\ln \lambda),
\end{align}
i.e. the homogeneous $\kappa$ scaling for $u(x,t)$ corresponds to the time shift for $\tilde{u}_\kappa(y,\tau)$ by $-2\ln\lambda$. The Weiss energy in the original coordinates is well-behaved with respect to the parabolic rescaling,
i.e. $W_\kappa(u_\lambda(t))=W_\kappa(u(\lambda^2 t))$.  In the conformal coordinates this leads to
\begin{align}
\label{eq:scaling_inv}
W_\kappa(\tilde{u}_\kappa(\tau - 2\ln\lambda))=W_\kappa(\tilde{u}_{\kappa,\lambda}(\tau)).
\end{align}

\section{Parabolic epiperimetric inequality}\label{sec:dynamical}
We describe a dynamical system approach for deriving the decay of the Weiss energy $W_{\kappa}(\tilde{u}(\tau))$ along solutions to \eqref{eq:bulk_conf}--\eqref{eq:boundary_conf} with $\kappa=3/2$ or $\kappa =2m$, $m\in \N_+$.

\subsection{The case $\kappa=\frac{3}{2}$}

In this case, stationary solutions to \eqref{eq:bulk_conf}--\eqref{eq:boundary_conf} are in $\mathcal{E}_{3/2}$ by Proposition \ref{prop:wlimit}. We will project our solution $\tilde{u}(\tau):=\tilde{u}_{3/2}(\tau)$ to $\mathcal{E}_{3/2}$ for each $\tau>0$, i.e. let
\begin{equation*}
\lambda(\tau)h_{e(\tau)}(y):=\lambda(\tau)h(e(\tau)\cdot y', y_n),
\end{equation*}
where $h(x_1,x_2):=c_n\Ree(x_1+i|x_2|)^{3/2}$, $c_n>0$ and $\|h\|_{L^2_\mu(\R^{n}_+)}=1$, be such that
\begin{align*}
\left\|\tilde{u}(\tau)-\lambda(\tau)h_{e(\tau)}\right\|_{L^2_\mu}=\text{dist}_{L^2_\mu}(\tilde{u}(\tau),\mathcal{E}_{3/2})=\inf_{\substack{\lambda\geq 0, \\
e\in \mathbb{S}^{n-1}\cap \{y_n=0\}}}\|\tilde{u}(\tau,\cdot)-\lambda h_e\|_{L^2_\mu},
\end{align*}
and study the evolution of
$\text{dist}_{L^2_\mu}(\tilde{u}(\tau,\cdot), \mathcal{E}_{3/2})$.

Due to the non-convexity of $\mathcal{E}_{3/2}$ the projection $\lambda h_e$ of $\tilde{u}(\tau, \cdot)$ onto $\mathcal{E}_{3/2}$ is not necessarily unique, hence the regularity of the parameters $\lambda, e$ in dependence of $\tau$ is in question. In particular this implies that we have to take care in our dynamical systems argument and can not directly work with the evolution equations for the parameters $\lambda(\tau)$ and $e(\tau)$. Instead, we rely on robust (energy type) identities for the Weiss energy.

To this end, we split $\tilde{u}$ into its leading order profile and an error:
\begin{align}\label{eq:decom}
\tilde{u}(y,\tau)=\lambda(\tau)h_{e(\tau)}(y)+\tilde{v}(y,\tau).
\end{align}
Here $\lambda(\tau)h_{e(\tau)}(y)$ is chosen such as to minimize the $L^2_{\mu}$ distance of $\tilde{u}$ to the set $\mathcal{E}_{3/2}$. We stress again, that this decomposition is a priori not necessarily unique.
From the minimality of $\|\tilde{u}(y,\tau)-\lambda(\tau)h_{e(\tau)}(y)\|_{L^2_{\mu}}$ we infer the following orthogonality conditions
\begin{align}
\lambda(\tau)\int\limits_{\R^n_+} h_{e(\tau)}(y)\tilde{v}(y,\tau) d\mu =0, \label{eq:orth1}
\end{align}
\begin{align}
\lambda(\tau)\int\limits_{\R^n_+} \Ree(y\cdot \tilde{e}(\tau)+i|y_n|)^{1/2}(y\cdot \tilde{e}(\tau)) \tilde{v}(y,\tau)d\mu &=0
\label{eq:orth22},\  \forall \tilde{e}\in \mathbb{S}^{n-1}, \tilde{e}\perp \text{span} \{e_n,e(\tau)\}.
\end{align}
Here we have used that $c\Ree(x_1+i|x_2|)^{1/2}=\p_{x_1}h(x_1,x_2)$. 

The next lemma concerns about the Weiss energy in terms of the error term $\tilde{v}$:
\begin{lemma}
The Weiss energy $W(\tilde{u}):=W_{3/2}(\tilde{u})$ can be written in terms of $\tilde{v}$ as
\begin{align}
\label{eq:weissv}
W(\tilde{u}) = W(\tilde{v}) - \frac{\lambda}{2} \int_{\{y_n=0\}} \tilde{u}\p_nh_e d\mu.
\end{align}
In particular, if $\tilde{u}\geq 0$ on $\{y_n=0\}$, then $W(\tilde{v})\leq W(\tilde{u})$.
\end{lemma}
\begin{proof}
To show \eqref{eq:weissv}, we observe that since $W(h_e)=0$,
\begin{align*}
W(\tilde{u})&=W(\tilde{v}) + \frac{2\lambda}{4}\int\limits_{\R^n_+} \nabla \tilde{v}\cdot \nabla h_e d\mu\\
&=W(\tilde{v}) -\frac{\lambda}{2} \int_{\{y_n=0\}} \tilde{v}\p_nh_e d\mu - \frac{\lambda}{2}\int\limits_{\R^n_+} \tilde{v}\di (e^{-|y|^2}\nabla h_e) dy.
\end{align*}
Using that $\di (e^{-|y|^2}\nabla h_e)=e^{-|y|^2}(\Delta h_e - 2y\cdot \nabla h_e)=- 3 h_e e^{-|y|^2}$, the orthogonality condition \eqref{eq:orth1} and that $\tilde{u}\p_n h_e=\tilde{v}\p_nh_e$ on $\{y_n=0\}$, we obtain \eqref{eq:weissv}.
\end{proof}

As our main auxiliary result, we deduce the following contraction argument, which is of the flavour of an epiperimetric inequality.

\begin{prop}\label{prop:epi}
Let $\tilde{u}: \R^n \times [0,\infty) \rightarrow \R$ be a solution of the parabolic thin obstacle problem \eqref{eq:bulk_conf}--\eqref{eq:boundary_conf} and satisfy \eqref{eq:sob_conf}.
Then there exists a constant $c_0\in (0,1)$ depending only on $n$, such that
\begin{align*}
W(\tilde{u}(\tau+1))\leq (1-c_0)W(\tilde{u}(\tau)) \text{ for any }\tau\in (0,\infty).
\end{align*}
\end{prop}

\begin{proof}
We argue by contradiction and use the contradiction assumption in combination with \eqref{eq:weissv} to derive enough compactness.\\

\emph{(i).} Assume that the statement were not true. Then there exists a sequence $\{c_j\}$ with $c_j\in (0,1/2)$, $c_j\rightarrow 0$, solutions $\tilde{u}_j$ and times $\tau_j$ such that
\begin{align}\label{eq:contra_assumption}
W(\tilde{u}_j(\tau_j+1))> (1-c_j)W(\tilde{u}_j(\tau_j)).
\end{align}
The contradiction assumption \eqref{eq:contra_assumption} implies that
\begin{align*}
W(\tilde{u}_j(\tau_j+1))-W(\tilde{u}_j(\tau_j))> \frac{-c_j}{1-c_j} W(\tilde{u}_j(\tau_j+1)).
\end{align*}
Using \eqref{eq:Weiss_new2} and the monotone decreasing property of $\tau\mapsto W(\tilde{u}(\tau))$ we infer
\begin{align}\label{eq:contra}
\int\limits_{\tau_j}^{\tau_j+1} \|\p_\tau \tilde{u}_j\|_{L^2_\mu}^2 \ d\tau < \frac{c_j}{2(1-c_j)} W(\tilde{u}_j(\tau_j+1))\leq c_j \int\limits_{\tau_j}^{\tau_j+1}W(\tilde{u}_j(\tau))\ d\tau.
\end{align}

\emph{(ii).} We project $\tilde{u}_j(\tau)$ to $\mathcal{E}_{3/2}$ for each $\tau>0$, and split $\tilde{u}_j$ into its leading order profile and an error term as in \eqref{eq:decom}:
\begin{align*}
\tilde{u}_j(\tau)=\lambda_j (\tau) h_{e_j(\tau)} +\tilde{v}_j(\tau).
\end{align*}
In this step we aim to derive the following upper bound for $W(\tilde{v}_j)$ from the contradiction assumption \eqref{eq:contra}:
\begin{equation}\label{eq:wbound}
\begin{split}
\int_{\tau_j}^{\tau_j+1}W(\tilde{v}_j(\tau))\ d\tau
&\leq 2c_j \int_{\tau_j}^{\tau_j+1}\|\tilde{v}_j(\tau)\|_{L^2_\mu}^2\ d\tau.
\end{split}
\end{equation}

First we observe that for any solution $\tilde{u}$ and any time interval $I=[\tau_a, \tau_b]$,
\begin{equation*}
\int\limits_I W(\tilde{u}(\tau))\ d\tau \stackrel{\eqref{eq:w2}}{=}-\int\limits_I\frac{1}{2}\p_{\tau} \|\tilde{u}\|_{L^2_\mu}^2\ d\tau =-\int\limits_{\R^n_+\times I} \lambda \p_\tau\tilde{u} h_e\ d\mu d\tau -\int\limits_{\R^n_+\times I} \p_\tau\tilde{u}\tilde{v}\ d\mu d\tau
\end{equation*}
Using that $\p_\tau\tilde{u}=\mathcal{L}\tilde{u}$ in $\R^n_+\times (0,\infty)$, where $\mathcal{L}:=\frac{1}{4}\Delta -\frac{y}{2}\cdot \nabla + \frac{3}{4}$, $\mathcal{L}h_e=0$ in $\R^n_+$ and an integration by parts we have
\begin{equation}\label{eq:w3232}
\begin{split}
\int\limits_I W(\tilde{u}(\tau))\ d\tau &=-\int\limits_{\R^n_+\times I} \p_\tau \tilde{u} \tilde{v}\ d\mu d\tau - \int\limits_{\R^n_{+}\times I} \mathcal{L}\tilde{u} (\lambda h_e) \ d\mu d\tau  \\
&=-\int\limits_{\R^n_+\times I} \p_\tau \tilde{u} \tilde{v}\ d\mu d\tau +\frac{\lambda}{4}\int\limits_{\{y_n=0\}} \left(\p_n \tilde{u}h_e - \tilde{u}\p_nh_e\right)\ d\mu d\tau.
\end{split}
\end{equation}
Next we apply \eqref{eq:w3232} to $\tilde{u}_j$ with $I_j:=[\tau_j,\tau_j+1]$. For the bulk integral we use H\"older and \eqref{eq:contra} to get
\begin{equation}\label{eq:dtauv}
\begin{split}
\left|\int_{\R^n_+\times I_j} \p_\tau \tilde{u}_j \tilde{v}_j\ d\mu d\tau \right|
\leq \left(c_j\int_{I_j} W(\tilde{u}_j(\tau)) d\tau \right)^{1/2}\left(\int_{I_j}\|\tilde{v}_j\|_{L^2_\mu}^2 d\tau \right)^{1/2}.
\end{split}
\end{equation}
Combining \eqref{eq:w3232} and \eqref{eq:dtauv} and using Young's inequality, we obtain
\begin{align*}
\frac{3}{4}\int\limits_{I_j} W(\tilde{u}_j(\tau))d\tau \leq c_j\int\limits_{I_j}\|\tilde{v}_j\|_{L^2_\mu}^2 d\tau + \int\limits_{I_j}\int\limits_{\{y_n=0\}} \frac{\lambda_j}{4}\left(\p_n \tilde{u}_jh_{e_j}-\tilde{u}_j\p_nh_{e_j}\right)d\mu d\tau.
\end{align*}
Recalling the relation between $W(\tilde{u})$ and $W(\tilde{v})$ in \eqref{eq:weissv}, we infer
\begin{align}
\label{eq:est_W_up}
\begin{split}
\frac{3}{4}\int\limits_{I_j} W(\tilde{v}_j(\tau))\ d\tau \leq c_j\int\limits_{I_j}\|\tilde{v}_j\|_{L^2_\mu}^2\ d\tau+ \int\limits_{I_j}\int\limits_{\{y_n=0\}}\left(\frac{\lambda_j}{4} \p_n\tilde{u}_j h_{e_j} + \frac{\lambda_j}{8} \tilde{u}_j\p_nh_{e_j}\right) \ d\mu d\tau.
\end{split}
\end{align}
Since, by the Signorini conditions, the second integral on the right hand side is less or equal to zero, we obtain the upper bound in \eqref{eq:wbound}. We remark that by rearrangement \eqref{eq:est_W_up} also entails that
\begin{align}
\label{eq:bound_term_ok}
\begin{split}
&- \frac{\lambda_j}{8} \int\limits_{I_j}\int\limits_{\{y_n=0\}}\left(\p_n\tilde{u}_j h_{e_j} +\tilde{u}_j\p_nh_{e_j} \right) d\mu d\tau\\
&\leq c_j \int\limits_{I_j} \|\tilde{v}_j\|_{L^2_{\mu}}^2\ d\tau - \frac{3}{4} \int\limits_{I_j} W(\tilde{v}_j(\tau))\ d\tau
 \leq \left(c_j + \frac{3}{4}\right)\int\limits_{I_j} \|\tilde{v}_j\|_{L^2_{\mu}}^2\ d\tau.
\end{split}
\end{align}

\vspace{10pt}

In the next steps \emph{(iii)-(iv)} we will use a compactness argument to arrive at a contradiction. The main idea is that one can find sequences $\hat \tau_j\in I_j$ and $\hat v_j(y):=\tilde{v}_j(\hat \tau_j,y)/\|\tilde{v}_j(\hat \tau_j)\|_{L^2_\mu}$, such that $\hat v_j$ converges to a nonzero blow-up profile in $\mathcal{E}_{3/2}$. This leads to a contradiction. The bounds on the Weiss energy for $\tilde{v}$ in step \emph{(ii)} are used to derive the desired compactness properties.\\

\emph{(iii).} We seek to prove that up to a subsequence
\begin{equation}\label{eq:hatu}
\hat{u}_j(y,\tau):=\frac{\tilde{u}_j(y,\tau_j+\tau)}{\|\tilde{u}_j\|_{L^2(I_j; L^2_\mu)}}, \ \tau\in [0,1]
\end{equation}
converges weakly in $L^2([0,1]; W^{1,2}_\mu)$, strongly in $C([0,1]; L^2_\mu)$ and locally in $C^1_xC^0_t$ up to $\{y_n=0\}$ to some stationary solution $\hat u_0$ to \eqref{eq:bulk_conf}--\eqref{eq:boundary_conf}.  Note that $\|\hat{u}_j\|_{L^2([0,1]; L^2_\mu)}=1$ by our normalization, thus the strong $L^2$-convergence implies that $\|\hat u_0\|_{L^2_\mu}=1$. By Proposition \ref{prop:wlimit} necessarily $\tilde{u}_0=c_n h(y'\cdot e_0,y_n)\in \mathcal{E}_{3/2}$ for some tangential direction $e_0$.

\emph{Proof for (iii).} First we note that by \eqref{eq:contra_assumption} and the monotone decreasing property of $\tau\mapsto W(\tilde{u}(\tau))$, we have $W(\tilde{u}_j(\tau))\leq 2\int_{I_j} W(\tilde{u}_j(\tau))d\tau$ for all $\tau\in I_j$. Then by \eqref{eq:weissv}, \eqref{eq:wbound} and \eqref{eq:bound_term_ok}, for some absolute constant $C>0$,
\begin{align}\label{eq:weissu}
W(\tilde{u}_j(\tau))\leq C \int_{I_j}\|\tilde{v}_j(\tau)\|_{L^2_\mu}^2d\tau\leq C \int_{I_j}\|\tilde{u}_j(\tau)\|_{L^2_\mu}^2d\tau, \quad \tau\in I_j.
\end{align}
Recalling the definition of the Weiss energy in \eqref{eq:Weiss_new} and the normalization \eqref{eq:hatu}, the above inequality can be rewritten as
\begin{equation*}
\sup_{\tau\in [0,1]}\frac{1}{4}\|\nabla \hat u_j(\tau)\|_{L^2_\mu}^2  \leq C+\frac{3}{4}.
\end{equation*}
Thus $\hat u_j\in L^\infty([0,1]; W^{1,2}_\mu)$. Next, \eqref{eq:weissu} together with \eqref{eq:contra} leads to
\begin{align}
\label{eq:iiia}
\int_{I_j}\|\p_\tau \tilde{u}_j\|_{L^2_\mu(\R^n_+)}^2d\tau \leq Cc_j\int_{I_j}\|\tilde{v}_j\|_{L^2_\mu}^2d\tau\leq Cc_j\int_{I_j}\|\tilde{u}_j\|_{L^2_\mu}^2d\tau,
\end{align}
which gives $\p_\tau \hat u_j \in L^2([0,1]; L^2_\mu)$ with $\|\p_\tau\hat u_j\|_{L^2([0,1]; L^2_\mu)}\leq Cc_j$. Since the embedding $W^{1,2}_\mu \hookrightarrow L^2_\mu$ is compact, by Aubin-Lions lemma, up to a subsequence $\hat u_j\rightarrow \hat u_0$ strongly in $C([0,1]; L^2_\mu)$ for some function $\hat u_0$. Note that $\hat u_j$ solves the variational inequality
\begin{equation*}
\begin{split}
\int\limits_0^1\int\limits_{\R^n_+} \left[ \p_\tau \hat{u}_j (v-\hat{u}_j)+ \frac{1}{4}\nabla \hat{u}_j\cdot \nabla (v-\hat{u}_j)-\frac{3}{4}\hat{u}_j(v-\hat{u}_j)\right] d\mu d\tau \geq 0,
\end{split}
\end{equation*}
for any $v\in L^2([0,1]; W^{1,2}_\mu)$, $v\geq 0 \text{ on } \{y_n=0\}$,
$v(0,\cdot)=\hat{u}_j(0,\cdot)$ and $v(\tau,\cdot)-\hat{u}_j(\tau,\cdot)$ has compact support in $\R^n_+$ for any $\tau\in [0,1]$. The interior regularity estimate for the solution to the variational inequality \cite{AU} entails that (after taking another subsequence) $\nabla \hat u_j\rightarrow \nabla \hat u_0$  locally in $C^{\alpha,\alpha/2}$. In the end using \eqref{eq:iiia} and passing to the limit in the variational inequality of $\hat u_j$ we conclude that $\hat u_0$ is a stationary solution.
\\

\emph{(iv).} Let $\tilde{u}_j$ be the sequence from step (iii) such that $\hat u_j\rightarrow \hat u_0\in \mathcal{E}_{3/2}$. Let $\hat \tau_j\in I_j$ be such that $\|\tilde{v}_j(\hat \tau_j)\|_{L^2_\mu}^2=\|\tilde{u}(\hat\tau_j)-\lambda_j(\hat\tau_j)h_{e(\hat\tau_j)}\|_{L^2_\mu}^2= \int_{I_j}\|\tilde{v}_j(\tau)\|_{L^2_\mu}^2 d\tau$. Consider
\begin{equation*}
\hat w_j(y,\tau):=\frac{\tilde{u}_j(y,\tau_j+\tau)-\lambda_j(\hat\tau_j) h_{e_j(\hat\tau_j)}(y)}{\|\tilde{v}_j(\hat\tau_j)\|_{L^2_\mu}},\ (y,\tau)\in \R^n_+\times [0,1].
\end{equation*}
We will prove that up to a subsequence $\hat w_j$ converge in $C([0,1];L^2_\mu)$ to a nonzero function  $\hat w_0 =\lambda_n h$ up to a rotation of coordinates, where $\lambda_n\neq 0$. This gives a contradiction. In fact, if $\lambda_n>0$, we get a contradiction because at each time step we have projected out $\mathcal{E}_{3/2}$ from $\tilde{u}_j$. If $\lambda_n<0$, then necessarily $\lambda_j(\hat \tau_j)=0$ for each $j$ (because otherwise $\int h_{e_j(\hat \tau_j)} \hat w_j d\mu=0$ by \eqref{eq:orth1}, which leads to a contradiction in the limit $j\rightarrow \infty$). However, it is a contradiction to \emph{Step (iii)}.

\emph{Proof for (iv).} Invoking \eqref{eq:weissu} and that
\begin{equation*}
\begin{split}
W(\tilde{u}_j(\tau)-\lambda_j(\hat\tau_j) h_{e_j(\hat\tau_j)})&=W(\tilde{u}_j(\tau))+\frac{\lambda_j(\hat\tau_j)}{2}\int_{\{y_n=0\}}\tilde{u}_j(\tau)\p_nh_{e_j(\hat\tau_j)} d\mu\\
&\leq W(\tilde{u}_j(\tau))
\end{split}
\end{equation*}
we have $W(\hat w_j(\tau))\leq Cc_j$ for any $\tau\in [0,1]$. As in step (iii) this implies that $\hat w_j$ is uniformly bounded in $L^\infty([0,1]; W^{1,2}_\mu)$.  By \eqref{eq:iiia},
\begin{equation}\label{eq:w}
\int_{I_j} \|\p_\tau\hat w_j\|_{L^2_\mu}^2 d\tau \leq Cc_j.
\end{equation}
Fundamental theorem of calculus together with \eqref{eq:w} gives that $\|\hat w_j(\tau)\|_{L^2_\mu}$ stays uniformly away from zero, i.e.
\begin{equation*}
\left|\|\hat{w}_j(\tau)\|_{L^2_\mu}^2- \|\hat{w}_j(\hat \tau_j)\|_{L^2_\mu}^2\right|=\left|\|\hat{w}_j(\tau)\|_{L^2_\mu}^2- 1\right|\leq Cc_j, \  \tau\in [0,1].
\end{equation*}
Thus by Aubin-Lions lemma, up to a subsequence $\hat w_j$ converges in $C^0([0,1];L^2_\mu)$ to $\hat w_0$ with $\|\hat w_0\|_{L^2_\mu}=1$. Since $\p_\tau\hat w_j-\mathcal{L}\hat w_j=0$ in $\R^n_+$, by the interior estimates $\hat w_j$ converges locally smoothly in $\R^n_+$.  This together with \eqref{eq:w} implies that $\hat w_0$ is stationary and it solves $\mathcal{L}\hat w_0=0$ in $\R^n_+$.
We claim that the limiting function $\hat{w}_0$ satisfies the Dirichlet-Neumann boundary condition
\begin{align}\label{eq:D-N}
\hat w_0=0 \text{ on } \Lambda_0 :=\{y_n=0,\ y'\cdot e_0\leq 0\}, \
 \p_n\hat w_0=0 \text{ on } \Omega_0 &:=\{y_n=0\}\setminus \Lambda_0,
\end{align}
where $e_0$ is the tangential direction from step (iii). With this at hand, $\hat w_0\in W^{1,2}_{\mu}$ solves the eigenvalue problem for $-\frac{1}{2}\Delta + y\cdot \nabla$ with the Dirichlet-Neumann boundary condition:
$$(-\frac{1}{2}\Delta + y\cdot \nabla )\hat w_0=\frac{3}{2}\hat w_0\text{ in }\R^n_+, \quad \hat w_0=0 \text{ on }\Lambda_0, \quad \p_n\hat w_0=0 \text{ on } \Omega_0.$$
 By the characterization of the eigenfunctions for the second Dirichlet-Neumann eigenvalue (cf. Appendix \ref{sec:appendix}), and after a rotation of coordinate (such that $e_0=e_{n-1}$),
$$\hat w_0(y)= \lambda_n h(y_{n-1}, y_n) + \sum_{i=1}^{n-2}\lambda_ix_ih_{1/2}(y_{n-1},y_n), \quad \lambda_i\in \R.$$
Here $h_{1/2}(y_{n-1},y_n):=c_n \Ree(y_{n-1}+i|y_n|)^{1/2}$. The orthogonality condition \eqref{eq:orth22} implies that $\lambda_i=0$ for $i=1,\cdots, n-2$. Thus $\lambda_n\neq 0$. This leads to a contradiction as argued at the beginning of \emph{Step (iv)}.\\

In the end, we verify \eqref{eq:D-N}. This is a consequence of the complementary boundary conditions satisfied by $\hat u_j$ and the uniform convergence. Indeed, given $U\Subset \Omega_0$, using $c_nh_{e_0}>c>0$ in $U$ and the uniform convergence of $\hat u_j$ to $c_nh_{e_0}$ in $\overline{U}\times [0,1]$ we have, for sufficiently large $j$ depending on $U$, $\hat u_j>0$ in $U\times [0,1]$. By the complementary condition in terms of $\tilde{u}_j$ we have  $\p_n \hat{u}_j=0$ in $U\times [0,1]$. Next since $e_j(\hat\tau_j)$ converges to $e_0$, which follows from the convergence of $\hat u_j$ to $c_nh_{e_0}$ in $C^0([0,1];L^2_\mu)$, one has $\p_n h_{e_j(\hat\tau)}=0$ in $U$ for $j$ sufficiently large. Thus $\p_n\hat w_j=0$ in $U\times [0,1]$ for sufficiently large $j$. Therefore, after extending $\hat w_j$ evenly about $\{y_n=0\}$, $\hat w_j$ solves $\p_t\hat w_j-\mathcal{L}\hat w_j=0$ in $\tilde{U}\times [0,1]$, where $\tilde{U}$ is an open neighborhood of $U$ in $\R^n$ with $\tilde{U}\cap \{y_n=0\}=U$. By the interior estimates $\hat w_j\rightarrow \hat w_0$ in $C^1(\tilde{U}\times [0,1])$. This implies that in the limit $\p_n\hat w_0=0$ on $U$. Since $U$ is arbitrary we have $\p_n\hat w_0=0$ on $\Omega_0$.
Using the fact that $c_n\p_n h_{e_0}<-c<0$ in $U\Subset \inte(\Lambda_0)$ and arguing similarly we can conclude that $\hat w_0=0$ on $\inte(\Lambda_0)$.
\end{proof}

A very similar but simpler argument as for Proposition \ref{prop:epi} gives the decay estimate of the Weiss energy if it becomes negative starting from some time $\tau_0$. After a shift in time, we may assume $\tau_0=0$.

\begin{prop}\label{prop:epi_neg}
Let $\tilde{u}: \R^n \times [0,\infty) \rightarrow \R$ be a solution to \eqref{eq:bulk_conf}--\eqref{eq:boundary_conf} with $\kappa=3/2$ and satisfy \eqref{eq:sob_conf}. Assume that $W(\tilde{u}(\tau))\leq 0$ for all $\tau>0$.
Then there exists a constant $c_0\in (0,1)$ depending only on $n$, such that
\begin{align*}
W(\tilde{u}(\tau+1))\leq (1+c_0)W(\tilde{u}(\tau)) \text{ for any }\tau\in (0,\infty).
\end{align*}
\end{prop}

As an immediate consequence of Proposition \ref{prop:epi} and Proposition \ref{prop:epi_neg}, one obtains the exponential decay rate of the Weiss energy:
\begin{cor}\label{cor:decay_rate_32}
Let $\tilde{u}$ be a solution to \eqref{eq:bulk_conf}--\eqref{eq:boundary_conf} with $\kappa=3/2$ and satisfy \eqref{eq:sob_conf}. Then there exists $\gamma_0\in (0,1)$ depending only on $n$ such that
\begin{itemize}
\item [(i)] If $W(\tilde{u}(\tau))\geq 0$ for all $\tau\in (0,\infty)$, then $W(\tilde{u}(\tau))\leq e^{-\gamma_0 \tau}W(\tilde{u}(0))$. Moreover, the limit $\lim_{\tau\rightarrow \infty}\tilde{u}(\tau)=:\tilde{u}(\infty)=\lambda(\infty)h_{e(\infty)}\in\mathcal{E}_{3/2}$ exists, and it satisfies
\begin{equation*}
\begin{split}
\|\tilde{u}(\tau)-\tilde{u}(\infty)\|_{L^2_\mu}^2 &\leq C_n W(\tilde{u}(0)) e^{-\gamma_0\tau}, \\
\|\tilde{v}(\tau)\|_{L^2_\mu}^2+|\lambda(\tau)^2-\lambda(\infty)^2| &\leq C_n W(\tilde{u}(0)) e^{-\gamma_0\tau}
\end{split}
\end{equation*}
for all $\tau\in (0,\infty)$.
\item [(ii)] If $W(\tilde{u}(0))<0$, then $W(\tilde{u}(\tau))\leq e^{\gamma_0 \tau} W(\tilde{u}(0))$. Moreover,
\begin{equation*}
\|\tilde{u}(\tau)\|_{L^2_\mu}^2\geq -\frac{2W(\tilde{u}(0))}{\gamma_0}(e^{\gamma_0\tau}-1)+\|\tilde{u}(0)\|_{L^2_\mu}^2.
\end{equation*}
\end{itemize}
\end{cor}
\begin{proof}
We only provide the proof for the case of nonnegative Weiss energy. The proof for (ii) is the same.

Assuming Proposition \ref{prop:epi} and arguing inductively we then obtain
\begin{equation*}
W(\tilde{u}(\tau+k))\leq (1-c_0)^k W(\tilde{u}(\tau))
\end{equation*}
for any $\tau\in (0,\infty)$ and $k\in \N_+$. This together with the monotonicity property of $\tau\mapsto W(\tilde{u}(\tau))$ implies that there exists $\gamma_0\in (0,1)$ depending only on $c_0$ such that
\begin{equation*}
W(\tilde{u}(\tau))\leq e^{-\gamma_0\tau}W(\tilde{u}(0)).
\end{equation*}
For any $0<\tau_1<\tau_2\leq \tau_1+1<\infty$, by \eqref{eq:Weiss_new2} and H\"older's inequality,
\begin{equation*}
\begin{split}
\|\tilde{u}(\tau_1)-\tilde{u}(\tau_2)\|_{L^2_\mu}&\leq \int_{\tau_1}^{\tau_2}\|\p_\tau\tilde{u}\|_{L^2_\mu} d\tau\\
&\leq \left(W(\tilde{u}(\tau_1))-W(\tilde{u}(\tau_2))\right)^{1/2}(\tau_2-\tau_1)^{1/2}\\
&\leq W(\tilde{u}(0))^{1/2}e^{-\gamma_0\tau_1/2}.
\end{split}
\end{equation*}
Here in the second last inequality we have used that $W(\tilde{u}(\tau_2))\geq 0$.
An iterative argument then yields for any $0<\tau_1<\tau_2<\infty$,
\begin{equation*}
\|\tilde{u}(\tau_1)-\tilde{u}(\tau_2)\|_{L^2_\mu}\leq \frac{W(\tilde{u}(0))^{1/2}}{1-e^{-\gamma_0/2}}e^{-\gamma_0\tau_1/2}.
\end{equation*}
Thus $\lim_{\tau\rightarrow \infty}\tilde{u}(\tau)=:\tilde{u}(\infty)\in L^2_\mu$ exists and the convergence rate is exponential. To show the exponential convergence of $\|\tilde{v}(\tau)\|_{L^2_\mu}=\dist_{L^2_\mu}(\tilde{u}(\tau),\mathcal{E}_{3/2})$ and  $\lambda(\tau)=\|\text{proj}_{L^2_\mu}(\tilde{u}(\tau),\mathcal{E}_{3/2})\|_{L^2_\mu}$ we only need to observe
\begin{equation*}
\begin{split}
\|\tilde{v}(\tau)\|_{L^2_\mu}^2 &\leq \|\tilde{u}(\tau)-\tilde{u}(\infty)\|_{L^2_\mu}^2,\ |\lambda(\tau)^2-\lambda(\infty)^2| \leq \|\tilde{v}(\tau)\|_{L^2_\mu}^2+ \|\tilde{u}(\tau)-\tilde{u}(\infty)\|_{L^2_\mu}^2,
\end{split}
\end{equation*}
and using the exponential convergence of $\tilde{u}(\tau)$ to $\tilde{u}(\infty)$.
\end{proof}

\begin{rmk}\label{rmk:non_zero}
At this stage it is possible that $\tilde{u}(\infty)$ is zero. However, if at the initial time
\begin{equation}\label{eq:smallness_32}
\begin{split}
W(\tilde{u}(0))\leq \delta_n \|\tilde{u}(0)\|_{L^2_\mu}^2, \quad
\dist_{L^2_\mu}(\tilde{u}(0),\mathcal{E}_{3/2})^2\leq \delta_n \|\tilde{u}(0)\|_{L^2_\mu}^2
\end{split}
\end{equation}
for some small $\delta_n>0$, then in the limit $\lambda(\infty)>0$. To see this, we note that the bound on the Weiss energy together with the exponential convergence of $\lambda(\tau)^2$ from Corollary \ref{cor:decay_rate_32} (i) implies
\begin{equation*}
|\lambda(0)^2-\lambda(\tau)^2|\leq C_n \delta_n \|\tilde{u}(0)\|_{L^2_\mu}^2, \text{ for all }\tau>0.
\end{equation*}
From the bound on the distance we have $\lambda(0)^2 \geq (1-\delta_n)\|\tilde{u}(0)\|_{L^2_\mu}^2$. Combining together leads to $\lambda(\tau)^2\geq (1-C_n\delta_n)\|\tilde{u}(0)\|_{L^2_\mu}^2$ for any $\tau>0$. Thus $\lambda(\infty)>\frac{1}{2}\|\tilde{u}(0)\|_{L^2_\mu}>0$ if $\delta_n$ is chosen sufficiently small.

We also note that \eqref{eq:smallness_32} is satisfied by requiring that the solution stays close to $\mathcal{E}_{3/2}$ in $W^{1,2}_\mu$ norm at $\tau=0$, i.e.
\begin{equation*}
\dist_{W^{1,2}_\mu}\left(\frac{\tilde{u}(\cdot, 0)}{\|\tilde{u}(\cdot, 0)\|_{L^2_\mu}},\mathcal{E}_{3/2}\right)\leq \delta_n.
\end{equation*}
\end{rmk}

\subsection{The case $\kappa=2m$.}
In this section we derive the decay estimate of the Weiss energy
\begin{align*}
W_{2m}(\tilde{u}(\tau))=\frac{1}{4}\int_{\R^n_+}|\nabla \tilde{u}(\tau)|^2 d\mu - m \int_{\R^n_+} |\tilde{u}(\tau)|^2 d\mu
\end{align*}
along solutions $\tilde{u}:=\tilde{u}_{2m}$ to the Signorini problem \eqref{eq:bulk_conf}--\eqref{eq:boundary_conf}.
Recall that when the frequency $\kappa=2m$, stationary solutions are in $\mathcal{E}_{2m}^+$ by Proposition \ref{prop:wlimit}, where $\mathcal{E}_{2m}^+$ is a subset of zero eigenspace of the Ornstein-Uhlenbeck operator $\mathcal{L}_{2m} :=\frac{1}{4}\Delta- \frac{y}{2}\cdot \nabla + m$.\\

The strategy is similar as for the case $\kappa=3/2$. For each $\tau$ we project our solution to the finite dimensional linear space $\mathcal{E}_{2m}$: 
$$\tilde{u}(\tau)=p(\tau)+\tilde{v}(\tau), \quad p(\tau)\in \mathcal{E}_{2m}.$$
Then we argue by contradiction that, if the associated Weiss energy $W_{2m}(\tau)$ does not decay fast enough, then after renormalization $\tilde{v}(\tau)$ would converge to a nonzero element in the eigenspace $\mathcal{E}_{2m}$ as $\tau$ goes to infinity. Different from the case $\kappa=3/2$, where the zero Dirichlet-Neumann eigenspace for $\mathcal{L}_{3/2}$ is the tangent space to the manifold generated by the \emph{unique} blow-up profile $\Ree(x+iy)^{3/2}$ together with rotation symmetry (cf. Appendix), in the case $\kappa=2m$ the zero eigenspace is not associated with a unique blow-up profile. That is one of the main reasons we choose to project out the whole linear space $\mathcal{E}_{2m}$ instead of the cone generated by a blow-up limit. Since functions in $\mathcal{E}_{2m}$ can change the sign on $\R^{n-1}\times\{0\}$, the estimates of the nonlinearity (which is concentrated on the unknown contact set $\{\tilde{u}=0\}$ at the boundary $\R^{n-1}\times \{0\}$) thus become more complicated in the case $\kappa=2m$.\\

More precisely, let $\{p_\alpha=c_\alpha H_{\alpha_1}(y_1)\cdots H_{\alpha_n}(y_n)\}_{\alpha\in \N^n}$ be the set of Hermite polynomials in $\R^n$, where $H_k$ for $k\in \N$ is the $1d$-Hermite polynomial of order $k$, i.e. it solves the eigenfunction equation $U''-2xU= -2k U$ in $\R$. Here $c_\alpha$ is chosen such that $\|p_\alpha\|_{L^2_\mu}=1$. Then $\{p_\alpha\}$ are eigenfunctions of $\mathcal{L}_{2m}$:
\begin{align*}
\mathcal{L}_{2m}p_\alpha= \left(m-\frac{|\alpha|}{2}\right)p_\alpha,
\end{align*}
and they form an orthonormal basis for $L^2_\mu$.
Let
$$\mathcal{E}_{2m}=\{\sum_{\alpha:|\alpha|=2m} \lambda_\alpha p_\alpha, \ p_\alpha(y',y_n)=p_\alpha(y',-y_n), \ \lambda_\alpha\in \R\}$$
be the subspace generated by $2m$-Hermite polynomials with symmetry. Given a solution $\tilde{u}(\tau)$ we consider the $L^2_\mu$ projection of $\tilde{u}(\tau)$ onto $\mathcal{E}_{2m}$
\begin{align}\label{eq:u-2m}
\tilde{u}(y, \tau)=\sum_{|\alpha|=2m}\lambda_\alpha(\tau)p_{\alpha}(y)+\tilde{v}(y,\tau), \quad \lambda_\alpha(\tau)=\int\limits_{\R^n_+} \tilde{u}(\tau,y)p_\alpha(y) d\mu
\end{align}
and study the evolution of $\dist_{L^2_\mu}(\tilde{u}(\tau), \mathcal{E}_{2m})=\|\tilde{v}(\tau)\|_{L^2_\mu}$ and the parameters $\lambda_\alpha(\tau)$. Due to the minimality, $\tilde{v}$ satisfies the orthogonality condition
\begin{align}\label{eq:orth2}
\int_{\R^n} \tilde{v} p_\alpha  d\mu=0, \text{ for any  } p_\alpha\in \mathcal{E}_{2m}.
\end{align}

Using the orthogonality condition and the equation of $\tilde{u}$, we have the following:

\begin{lemma}\label{lem:evo_parameter}
Let $\tilde{u}(\tau)$ be a solution to \eqref{eq:bulk_conf}--\eqref{eq:boundary_conf} which satisfies \eqref{eq:sob_conf}. Consider the orthogonal decomposition as in \eqref{eq:u-2m}. Then 
\begin{itemize}
\item[(i)] For a.e. $\tau$, the parameters $\lambda_\alpha(\tau)$ and $\|\tilde{v}(\tau)\|_{L^2_\mu}$ satisfy the evolutionary equations
\begin{align}
\frac{1}{2}\p_\tau \|\tilde{v}(\tau)\|_{L^2_\mu}^2 &=-W_{2m}(\tilde{v}(\tau))+\sum_{|\alpha|=2m} \frac{\lambda_\alpha(\tau)}{4} \int\limits_{\{y_n=0\}} p_\alpha \p_n\tilde{v}(\tau)\ d\mu,\label{eq:weiss2m}\\
\dot{\lambda}_\alpha(\tau) & = -\frac{1}{4}\int\limits_{\{y_n=0\}} p_\alpha \p_n \tilde{v}(\tau)\ d\mu,\label{eq:lambda2m}
\end{align}
for each multi-index $\alpha$ with $|\alpha|=2m$.
\item[(ii)] For any $0<\tau_a<\tau_b<\infty$,
\begin{equation}\label{eq:dtauW2}
\begin{split}
W_{2m}(\tilde{u}(\tau_b))-W_{2m}(\tilde{u}(\tau_a))&\leq -2 \int\limits_{\tau_a}^{\tau_b}(\|\p_\tau \tilde{v}\|_{L^2_\mu}^2 + \sum_{|\alpha|=2m} \dot{\lambda}_\alpha^2 )\ d\tau.\\
\end{split}
\end{equation}
\item[(iii)] For each $\tau$, the Weiss energy of $\tilde{v}$ is the same as for $\tilde{u}$:
\begin{align}\label{eq:weiss_equal}
W_{2m}(\tilde{u}(\tau))=W_{2m}(\tilde{v}(\tau)).
\end{align}
\end{itemize}
\end{lemma}
\begin{proof}
(i) Note that since $\p_\tau\tilde{u}\in L^2_{loc}(\R_+;L^2_\mu)$, we have that $\dot{\lambda}_\alpha\in L^2_{loc}(\R_+)$.
From the equation of $\tilde{u}$ and that $\mathcal{L}_{2m}p_\alpha=0$, $\tilde{v}$ satisfies
\begin{align*}
\p_{\tau}\tilde{v}=\mathcal{L}_{2m}\tilde{v}-\sum_{|\alpha|=2m}\dot{\lambda}_\alpha p_\alpha \text{ in }\R^{n}_+\times (0,\infty)
\end{align*}
with the Signorini condition
\begin{align*}
\sum_{|\alpha|=2m} \lambda_\alpha p_\alpha \p_n\tilde{v} + \tilde{v}\p_n\tilde{v} =0 \text{ on } \{y_n=0\}.
\end{align*}
Multiplying $\tilde{v}$ on both sides of the equation, using the Signorini condition and the orthogonality \eqref{eq:orth2} we obtain \eqref{eq:weiss2m}.
Multiplying $p_\alpha$ on both sides of the equation,  using the orthogonality condition (which gives $\int_{\R^n} \p_\tau \tilde{v} p_\alpha d\mu=0$ for a.e. $\tau$) we get the evolution equation for $\lambda_\alpha$ in \eqref{eq:lambda2m}. 

(ii) The estimate \eqref{eq:dtauW2} follows from \eqref{eq:Weiss_new} and orthogonality \eqref{eq:orth2}.

(iii) Using $W_{2m}(p_\alpha)=0$ for $|\alpha|=2m$ we have
\begin{align*}
W_{2m}(\tilde{u})= W_{2m}(\tilde{v}) + \frac{1}{4}\sum_\alpha \lambda_\alpha \int\limits_{\R^n_+} \nabla \tilde{v}\cdot \nabla p_\alpha d\mu.
\end{align*}
An integration by parts, $\p_np_\alpha=0$ on $\{y_n=0\}$ and \eqref{eq:orth2} yield that the last term is zero. Thus we obtain \eqref{eq:weiss_equal}.
\end{proof}

In the sequel, we will frequently use the following auxiliary function.
Let $h_{2m}$ denote the $0$-eigenfunction of $\mathcal{L}_{2m}$, which has the expression
\begin{equation}\label{eq:h2k}
h_{2m}(y)= C_{m,n}^{-1} \left(\sum_{j=1}^{n-1}2^{2m}\Ree(y_j+iy_n)^{2m}+m!\sum_{\ell=0}^m \frac{(-1)^\ell }{(m-\ell)!(2\ell)!}(2y_n)^{2\ell}\right),
\end{equation}
where $C_{m,n}\sim c_n 2^{2m}2^m m!$, $c_n>0$, is a normalization factor such that $\|h_{2m}\|_{L^2_\mu}=1$. Note that
$$h_{2m}(y',0)=C_{m,n}^{-1}(2^{2m}|y'|^{2m}+1).$$
In the sequel we will denote
$$\lambda_{2m}:=\int_{\R^n_+} \tilde{u} h_{2m} d\mu.$$
Since $\p_n\tilde{v}=\p_n\tilde{u}\leq 0$ and $h_{2m}\geq 0$ on $\{y_n=0\}$, we see from  \eqref{eq:lambda2m} that $\dot{\lambda}_{2m}\geq 0$.\\

The first proposition concerns the evolution of the Weiss energy $W_{2m}(\tilde{u}(\tau))$ if it is negative.

\begin{prop}\label{prop:negative_Weiss}
Let $\tilde{u}$ be a solution to the Signorini problem \eqref{eq:bulk_conf}--\eqref{eq:boundary_conf} with $\kappa=2m$ and satisfy \eqref{eq:sob_conf}. Assume that $W_{2m}(\tilde{u}(0))\leq 0$. Then there exists a constant $c_0\in (0,1)$ depending on $m,n$ such that
\begin{equation*}
W_{2m}(\tilde{u}(\tau+1))\leq (1+c_0)W_{2m}(\tilde{u}(\tau)), \quad \text{for any } \tau\in (0,\infty).
\end{equation*}
\end{prop}
\begin{proof}
Assume it were not true, then there exists a sequence of solutions $\tilde{u}_j$, $\tau_j\in (0,\infty)$ and $\epsilon_j\rightarrow 0$ such that
\begin{align*}
W_{2m}(\tilde{u}_j(\tau_j+1))\geq (1+\epsilon_j)W_{2m}(\tilde{u}_j(\tau_j)).
\end{align*}
For the rest of the proof we drop the dependence on $j$ for simplicity.
We decompose $\tilde{u}(\tau,\cdot)$ into $\tilde{u}(\tau, \cdot)=p_{<2m}(\tau, \cdot)+ p(\tau, \cdot) + \tilde{w}(\tau,\cdot)$, where $p\in \mathcal{E}_{2m}$, and $p_{<2m}(\tau, y)=\sum_{|\alpha|<2m}\lambda_{\alpha}(\tau)p_{\alpha}(y)$ is the projection of $\tilde{u}$ to the subspace $\mathcal{E}_{<2m}$ generated by $k$-Hermite polynomials $k<2m$ with symmetry, i.e.  $$\mathcal{E}_{<2m}:=\{\sum_{\alpha:|\alpha|<2m} c_\alpha p_\alpha, \ p_\alpha(y',y_n)=p_\alpha(y',-y_n)\}.$$ Note that
\begin{equation*}
W_{2m}(p_{<2m})\leq 0, \quad W_{2m}(\tilde{w})\geq 0.
\end{equation*}
Thus the contradiction assumption implies that
\begin{equation*}
\epsilon_j W_{2m}(\tilde{w}(\tau_j)) -\left[W_{2m}(\tilde{u}(\tau_{j}+1))-W_{2m}(\tilde{u}(\tau_j))\right] \leq -\epsilon_j W_{2m}(p_{<2m}(\tau_j)).
\end{equation*}
This together with \eqref{eq:dtauW2} and the monotone decreasing property of $\tau\mapsto W_{2m}(\tilde{u}(\tau))$ implies that
\begin{equation}\label{eq:contra_2m}
2\int\limits_{I_j}\sum_{\alpha:|\alpha|\leq 2m} \dot{\lambda}_\alpha(\tau)^2 d\tau \leq -\epsilon_j \int\limits_{I_j} W_{2m}(p_{<2m}(\tau)) d\tau, \quad I_j:=[\tau_j,\tau_{j}+1].
\end{equation}
We will show that \eqref{eq:contra_2m} implies for some $C=C(m,n)>0$
\begin{equation}\label{eq:w_neg}
\int\limits_{I_j}W_{2m}(p_{<2m}(\tau))d\tau \geq - C \epsilon_j\int\limits_{I_j}\|p_{<2m}(\tau)\|_{L^2_\mu}^2 d\tau.
\end{equation}
However, from the spectral gap we have, for any $\tau>0$
$$W_{2m}(p_{<2m}(\tau))=\sum_{|\alpha|<2m, |\alpha|\in \N} -\left(m-\frac{|\alpha|}{2}\right)\|p_\alpha(\tau)\|^2_{L^2_\mu}\leq - \|p_{<2m}(\tau)\|_{L^2_\mu}^2.$$
This is a contradiction to \eqref{eq:w_neg} if $\epsilon_j$ is sufficiently small. \\

It remains to prove \eqref{eq:w_neg}.
Multiplying the equation of $\tilde{u}$ by $p_{<2m}(\tau)$ and an integration by parts in space yield
\begin{equation*}
\int\limits_{I_j} W_{2m}(p_{<2m}(\tau)) \ d\tau =- \frac{1}{2}\int\limits_{I_j}\p_\tau\|p_{<2m}(\tau)\|_{L^2_\mu}^2 \ d\tau -\frac{1}{4}\int\limits_{I_j}\int\limits_{\{y_n=0\}}p_{<2m}(\tau)\p_n\tilde{u}(\tau)\ d\mu d\tau.
\end{equation*}
By \eqref{eq:contra_2m} the first integral can be estimated from below as
\begin{equation*}
\begin{split}
&-\frac{1}{2}\int\limits_{I_j}\p_\tau\|p_{<2m}(\tau)\|_{L^2_\mu}^2\ d\tau =-\int\limits_{I_j}\sum_{|\alpha|<2m} \lambda_\alpha\dot{\lambda}_\alpha \ d\tau \\
&\geq - \epsilon_j^{1/2}\left(\int\limits_{I_j}\|p_{<2m}(\tau)\|_{L^2_\mu} d\tau\right)^{1/2}\left(\int\limits_{I_j}-W_{2m}(p_{<2m}(\tau))d\tau\right)^{1/2}.
\end{split}
\end{equation*}
To estimate the boundary integral we observe that for each $\tau>0$
\begin{equation*}
\sup_{y'\in \R^{n-1}} \frac{|p_{<2m}(y',0,\tau)|}{h_{2m}(y')}\leq c_{n,m}\|p_{<2m}(y,\tau)\|_{L^2_\mu}.
\end{equation*}
Using $\p_n\tilde{u}(y',0)\leq 0$  and recalling the expression of $\dot\lambda_{2m}$ in \eqref{eq:lambda2m}, we can estimate the boundary term from below by
\begin{equation*}
\begin{split}
-\frac{1}{4}\int\limits_{I_j}\int_{\{y_n=0\}}p_{<2m}\p_n\tilde{u}\ d\mu d\tau&\geq\frac{c_{m,n}}{4}\int\limits_{I_j}\|p_{<2m}(\tau)\|_{L^2_\mu}\int\limits_{\{y_n=0\}} h_{2m}\p_n\tilde{u}\ d\mu d\tau\\
&=-c_{m,n} \int\limits_{I_j}\|p_{<2m}(\tau)\|_{L^2_\mu}\dot\lambda_{2m}(\tau) \ d\tau.
\end{split}
\end{equation*}
Invoking \eqref{eq:dtauW2} again we thus have
\begin{equation*}
\begin{split}
-\frac{1}{4}\int\limits_{I_j}\int\limits_{\{y_n=0\}}p_{<2m}\p_n\tilde{u}\ d\mu d\tau \geq -c_{m,n}\epsilon_j^{1/2} \left(\int\limits_{I_j}\|p_{<2m}(\tau)\|_{L^2_\mu}d\tau \right)^{1/2}\left(-\int\limits_{I_j}W_{2m}(p_{<2m}(\tau))d\tau\right)^{1/2}.
\end{split}
\end{equation*}
Combining together we obtain \eqref{eq:w_neg}, and the proof is complete.
\end{proof}

In the next proposition we derive a discrete logarithmic decay of the Weiss energy under the assumption that $W_{2m}(\tilde{u}(\tau))> 0$ for all $\tau$.

\begin{prop}\label{prop:homo2m2}
Let $\tilde{u}$ be a solution to \eqref{eq:bulk_conf}--\eqref{eq:boundary_conf} with $\kappa=2m$ and satisfy \eqref{eq:sob_conf}. Assume that $W_{2m}(\tilde{u}(\tau))> 0$ for all $\tau\geq 0$ and that
\begin{equation*}
W_{2m}(\tilde{u}(0))\leq \delta_0,
\end{equation*}
for some $\delta_0=\delta_0(m,n)>0$ small and $\|\tilde{u}(0)\|_{L^2_\mu}=1$. Then there exists a $c_0\in (0,1)$ depending only on $n$ and $m$ such that
\begin{equation*}
W_{2m}(\tilde{u}(\tau+1))\leq \left(1-c_0W_{2m}(\tilde{u}(\tau+1))\left|\ln W_{2m}(\tilde{u}(\tau+1))\right|^{2}\right)W_{2m}(\tilde{u}(\tau))
\end{equation*}
for all $\tau >0$.
\end{prop}

\begin{proof}
(i) Assume that the statement were wrong, then there exists a sequence of solutions $\tilde{u}_j$ with $\dist_{L^2_\mu}(\tilde{u}(0),\mathcal{E}_{2m}^+)\leq \delta_0$,  a sequence of positive constants $\epsilon_j\rightarrow 0$ and $\tau_j>0$, such that
\begin{equation*}
W_{2m}(\tilde{u}_j(\tau_j+1))\geq \left(1-\epsilon_j W_{2m}(\tilde{u}_j(\tau_j+1))\left|\ln W_{2m}(\tilde{u}_j(\tau_j+1))\right|^{2}\right)W_{2m}(\tilde{u}_j(\tau_j)).
\end{equation*}
Thus after rearranging the terms and using \eqref{eq:dtauW2} as well as that $\tau\mapsto W_{2m}(\tilde{u}_j(\tau))$ is monotone decreasing, we have
\begin{multline}\label{eq:contra2m2}
2\int\limits_{I_j} (\|\p_\tau \tilde{v}_j\|_{L^2_\mu}^2 + \sum_{|\alpha|=2m} \dot{\lambda}_\alpha^2) d\tau \leq W_{2m}(\tilde{u}_j(\tau_j))-W_{2m}(\tilde{u}_j(\tau_j+1))\\
\leq 2\epsilon_j  W_{2m}(\tilde{u}_j(\tau))^2 \left|\ln W_{2m}(\tilde{u}_j(\tau))\right|^{2} \ \text{ for a.e. }\tau\in I_j:=[\tau_j,\tau_{j}+1].
\end{multline}

\vspace{10pt}

(ii). For simplicity we denote $p_j(\tau, \cdot):=\sum_{|\alpha|=2m}\lambda^j_\alpha(\tau) p_\alpha(\cdot)\in \mathcal{E}_{2m}$, which is the projection of $\tilde{u}_j(\tau)$ onto $\mathcal{E}_{2m}$.
In the light of \eqref{eq:weiss2m}, to show the decay estimate of the Weiss energy we mainly need to estimate the boundary integral
\begin{equation*}
-\int\limits_{\R^{n-1}\times \{0\}} p_j(\tau) \p_n\tilde{v}_j(\tau) d\mu.
\end{equation*}
Let $(p_j)_-:=\max\{-p_j, 0\}$. We aim to show that there exists $C=C(m,n)$ such that
\begin{equation}\label{eq:M}
\begin{split}
-\int\limits_{I_j}\int\limits_{\R^{n-1}\times \{0\}}(p_j)_-(\tau)\p_n\tilde{v}_j(\tau) d\mu d\tau \leq C\epsilon_j^{1/2}\int\limits_{I_j}\left(\|\nabla \tilde{v}_j\|_{L^2_\mu}^2 +\|\tilde{v}_j\|_{L^2_\mu}^2\right)\ d\tau
\end{split}
\end{equation}
Due to the non-compactness of $\R^n_+$, the proof for \eqref{eq:M} is more involved than that of \cite{CSV17}. We will proceed with the rest of the proof assuming \eqref{eq:M}, and prove \eqref{eq:M} in the end. \\

(iii). Assuming \eqref{eq:M} we can estimate the Weiss energy for $\tilde{v}_j$ from above: 
\begin{equation*}
\begin{split}
\int\limits_{I_j}W_{2m}(\tilde{v}_j(\tau))\ d\tau &\stackrel{\eqref{eq:weiss2m}}{\leq} -\int\limits_{I_j}\int\limits_{\R^n_+} \tilde{v}_j\p_\tau \tilde{v}_j\ d\mu d\tau -\frac{1}{4}\int\limits_{I_j}\int\limits_{\R^{n-1}\times\{0\}}(p_j)_-\p_n\tilde{v}_j\ d\mu d\tau\\
&\stackrel{\eqref{eq:contra2m2}}{\leq} \int\limits_{I_j} \|\tilde{v}_j\|_{L^2_\mu}\|\p_{\tau}\tilde{v}_j\|_{L^2_\mu}\ d\tau+ C\epsilon_j^{1/2}\int\limits_{I_j}\left(\|\nabla \tilde{v}_j\|_{L^2_\mu}^2 +\|\tilde{v}_j\|_{L^2_\mu}^2\right) d\tau.
\end{split}
\end{equation*}
By H\"older's inequality,  \eqref{eq:contra2m2} and Cauchy-Schwartz, for $\delta_0$ sufficiently small (such that for all $\tau>0$, $W_{2m}(\tilde{v}_j(\tau))\left|\ln W_{2m}(\tilde{v}_j(\tau))\right|\leq 
W_{2m}(\tilde{v}_j(\tau))^{1/2}$) we have,
\begin{align*}
\int\limits_{I_j} \|\tilde{v}_j\|_{L^2_\mu}\|\p_{\tau}\tilde{v}_j\|_{L^2_\mu}
\ d\tau \leq C\epsilon_j^{1/2}\int\limits_{I_j}\|\tilde{v}_j\|_{L^2_\mu}^2 \ d\tau+ C\epsilon_j^{1/2}\int\limits_{I_j} W_{2m}(\tilde{v}_j(\tau))\ d\tau.
\end{align*}
Recalling the definition of the Weiss energy and rearranging the terms we have
\begin{equation*}
\left(\frac{1}{4}-C\epsilon_j^{1/2}\right)\int\limits_{I_j}\|\nabla \tilde{v}_j\|_{L^2_\mu}^2\ d\tau \leq \left(m+C\epsilon_j^{1/2}\right)\int\limits_{I_j}\|\tilde{v}_j\|_{L^2_\mu}^2\ d\tau.
\end{equation*}
Since $C$ is only depending on $m,n$, for $j$ sufficiently large the above inequality implies
\begin{equation}\label{eq:nablav_2m}
\int\limits_{I_j}\|\nabla \tilde{v}_j\|_{L^2_\mu}^2\ d\tau \leq 8m\int\limits_{I_j}\|\tilde{v}_j\|_{L^2_\mu}^2\ d\tau,
\end{equation}
and thus by \eqref{eq:contra2m2}
\begin{equation}\label{eq:dtw_2m}
\int\limits_{I_j} \|\p_\tau\tilde{v}_j\|^2_{L^2_\mu}\ d\tau \leq C\epsilon_j\int\limits_{I_j}\|\tilde{v}_j\|_{L^2_\mu}^2\ d\tau.
\end{equation}

\vspace{10pt}

\emph{(iv).} Consider $\tilde{w}_j(y,\tau):=\frac{\tilde{v}_j(y,\tau_j+\tau)}{\|\tilde{v}_j\|_{L^2(I_j;L^2_\mu)}}$, $(y,\tau)\in \R^n_+\times [0,1]$, which satisfies
\begin{equation*}
\p_\tau\tilde{w}_j =\mathcal{L}_{2m} \tilde{w}_j \text{ in } \R^n_+\times (0,1],\quad \p_n\tilde{w}_j\leq 0 \text{ on } \{y_n=0\}.
\end{equation*}
With \eqref{eq:contra2m2}, \eqref{eq:nablav_2m} and \eqref{eq:dtw_2m} at hand and arguing as in (iii) of Proposition \ref{prop:epi} we have that $\tilde{w}_j\in L^\infty([0,1]; W^{1,2}_\mu)$ and $\p_\tau\tilde{w}_j\in L^2([0,1];L^2_\mu)$. Up to a subsequence, $\tilde{w}_j$ converges weakly in $L^2([0,1];W^{1,2}_\mu)$ and strongly in $C([0,1];L^2_\mu)$ to a nonzero function $\tilde{w}_0$. The convergence is locally $C^\infty$ in $\R^n_+\times (0,1)$ by using the interior estimate of the equation.  By \eqref{eq:dtw_2m} and the equation for $\tilde{w}_j$, $\tilde{w}_0$ solves the stationary equation $\mathcal{L}_{2m}\tilde{w}_0=0$ in $\R^n_+$ and after an even reflection about $\{y_n=0\}$ satisfies $\mathcal{L}_{2m}\tilde{w}_0\leq 0$ in $\R^n$. By Proposition \ref{prop:wlimit} (or Lemma 12.4 in \cite{DGPT} in the conformal coordinates), we conclude that $\tilde{w}_0\in \mathcal{E}_{2m}$. This is a contradiction.
\\

In the end, we prove \eqref{eq:M}. 
We will divide the proof into three parts (a)-(c). Since the estimate is trivial if $(p_j)_-=0$, in the sequel we assume that $(p_j)_-$ is not identically zero on $\{y_n=0\}$. For a.e. $\tau$ we denote
$$c_\ast:=c_\ast(j,\tau):=\sup_{\R^{n-1}\times \{0\}}\frac{(p_j)_-(\tau)}{h_{2m}(y')},$$
where $h_{2m}(y')\sim_{n,m} \left(|y'|^{2m}+1\right)$ by \eqref{eq:h2k}. 
Noticing that
\begin{equation}\label{eq:cstar}
\begin{split}
&-\int\limits_{\R^{n-1}\times \{0\}} (p_j)_- \p_n\tilde{v}_j d\mu =-\int\limits_{\R^{n-1}\times \{0\}} \frac{(p_j)_-}{h_{2m}} h_{2m} \p_n\tilde{v}_j d\mu\\
&\leq c_\ast \int\limits_{\R^{n-1}\times \{0\}}-h_{2m}\p_n\tilde{v}_j d\mu=4 c_\ast \dot{\lambda}_{2m},
\end{split}
\end{equation}
and that $\dot\lambda_{2m}$ is related to the Weiss energy via \eqref{eq:contra2m2}, we mainly need to estimate $c_\ast$.

\emph{(a).} We show that if
\begin{equation*}
M:=\max_{\overline{B'_{R_0}}} \frac{(p_j)_-}{h_{2m}}\geq \frac{1}{2}c_\ast,
\end{equation*}
where $B'_{R_0}\subset \R^{n-1}\times \{0\}$ with $R_0^2:=R_0(j,\tau)^2:=-\frac{1}{2}\ln W_{2m}(\tilde{v}_j(\tau))$,
then there exists a positive constant $C=C(m,n)$ such that
\begin{equation*}
c_\ast \leq C\left(\|\tilde{v}_j\|_{L^2_\mu}^2+\|\nabla \tilde{v}_j\|_{L^2_\mu}^2\right)^{\frac{1}{n+1}}\left(W_{2m}(\tilde{v}_j)\right)^{-\frac{1}{2(n+1)}}.
\end{equation*}

Indeed, let $y_0\in \overline{B'_{R_0}}$ a point which realizes the maximum.
Since $\|p_j(\tau)\|_{L^2_\mu}^2\leq \|\tilde{u}_j(0)\|_{L^2_\mu}^2=1$, there exists $C=C(m,n)>0$ such that
$$M\leq C, \quad L:=[(p_j)_-/h_{2m}]_{\dot{C}^{0,1}(\{y_n=0\})}\leq C.$$
Let $r_0:=M/(c_nL)>0$. In $B'_{r_0}(y_0)\cap B'_{R_0}$ we have $(p_j)_-/h_{2m} \geq M/2$. 
Therefore, there exists a constant $C>0$ depending only on $m, n$ such that
\begin{equation*}
\begin{split}
M \leq c_n L^{\frac{n-1}{n+1}} \left(\int\limits_{B'_{r_0}(y_0)\cap B'_{R_0}} \left|\frac{(p_j)_-}{h_{2m}}\right|^2 dy' \right)^{\frac{1}{n+1}}\leq C e^{\frac{R_0^2}{n+1}}\left(\int\limits_{\R^{n-1}\times \{0\}} \left|\frac{(p_j)_-}{h_{2m}}\right|^2 e^{-|y'|^2} dy' \right)^{\frac{1}{n+1}}.
\end{split}
\end{equation*}
Since $h_{2m}$ is uniformly bounded away from zero and recalling our choice of $R_0$, we thus have
\begin{equation*}
M\leq C e^{\frac{R_0^2}{n+1}}\|(p_j)_-\|_{L^2_\mu(\R^{n-1}\times \{0\})}^{\frac{2}{n+1}}.
\end{equation*}
Using the Signorini condition $\tilde{u}_j=p_j+\tilde{v}_j\geq 0$ on $\{y_n=0\}$ we have
\begin{equation*}
\|(p_j)_-\|^2_{L^2_\mu(\R^{n-1}\times \{0\})}\leq \|\tilde{v}_j\|^2_{L^2_\mu(\R^{n-1}\times \{0\})}\leq C_n \left(\|\nabla \tilde{v}_j\|^2_{L^2_\mu}+\|\tilde{v}_j\|^2_{L^2_\mu}\right),
\end{equation*}
where the second inequality follows from the trace lemma. Combining the above two inequalities, using the definition of $R_0$ as well as the relation $c_\ast\leq 2M$, we complete the proof for \emph{(a)}.

\emph{(b).} We show that if $c_\ast >2M$, then there exists $C=C(m,n)>0$ such that
\begin{equation*}
c_\ast \leq \frac{C}{R_0^2}=-\frac{2C}{ \ln ( W_{2m}(\tilde{v}_j))}.
\end{equation*}

First we note that the statement is obvious if $\deg p_j(y',0)\leq 2m-2$. If $\deg p_j(y',0)=2m$, then by the even symmetry in $y_n$
$$p_j =p_j^{H}+ p_j^{R} \text{ on } \R^{n-1}\times \{0\},$$
where $p_j^H$ is the $2m$-homogeneous part of $p_j(\cdot, 0)$ and $\deg p_j^R\leq 2m-2$. Let $y_0\in \R^{n-1}\times \{0\}\cup\{\infty\}\setminus B'_{R_0}$ be such that $((p_j)_-/h_{2m})(y_0)=c_\ast$. If $(p_j^H/h_{2m})(y_0)\geq 0$, then $c_\ast \leq - (p_j^R/h_{2m})(y_0)\leq C/R_0^2$. If $(p_j^H/h_{2m})(y_0)<0$, then $\left|(p_j^R/h_{2m})(t y_0)\right|\geq \frac{1}{10}c_\ast$ for $t\in \R$ with $|ty_0|\geq R_0/2$. In fact, if it were not true, then necessarily $(p_j^H/h_{2m})(y_0)\leq -\frac{9}{10}c_\ast$. Furthermore, by choosing $\delta_0=\delta_0(m,n)$ sufficiently small (such that $R_0$ is sufficiently large) we have that $(p_j^H/h_{2m})(ty_0)\leq -\frac{8}{10}c_\ast$ for $|ty_0|\geq R_0/2$. Thus $(p_j/h_{2m})(ty_0)\leq -\frac{8}{10}c_\ast+\frac{1}{10}c_\ast \leq -\frac{7}{10}c_\ast$ for $|ty_0|\geq R_0/2$, which contradicts to the fact that $\sup_{B'_{R_0}} ((p_j)_-/h_{2m})\leq c_\ast/2$. 

\emph{(c).} Combining \emph{(a)-(b)} with \eqref{eq:cstar} and integrating in $\tau$ over $I_j$ yield
\begin{align*}
\int\limits_{I_j}\int\limits_{\R^{n-1}\times \{0\}}-\p_n\tilde{u}_j(p_j)_-\ d\mu &\leq C\int\limits_{I_j}\left(\|\nabla \tilde{v}_j\|_{L^2_\mu}^2+\|\tilde{v}_j\|_{L^2_\mu}^2\right)^{\frac{1}{n+1}}W_{2m}(\tilde{v}_j)^{-\frac{1}{2(n+1)}}\dot\lambda_{2m}\ d\tau\\
&+C\int\limits_{I_j} \left|\ln W_{2m}(\tilde{v}_j)\right|^{-1}\dot\lambda_{2m} \ d\tau=: \text{I}+\text{II}.
\end{align*}
By Cauchy-Schwartz, Jensen's inequality, the monotone decreasing property of $\tau\mapsto W_{2m}(\tilde{v}_j(\tau))$ and \eqref{eq:contra2m2} for $\dot{\lambda}_{2m}$,
\begin{align*}
\text{I}\leq C\epsilon_j^{1/2}\left(\int_{I_j}\|\nabla \tilde{v}_j\|_{L^2_\mu}^2+\|\tilde{v}_j\|_{L^2_\mu}^2 d \tau\right)^{\frac{1}{n+1}}W_{2m}(\tilde{v}_j(\tau_j+1))^{1-\frac{1}{2(n+1)}}\left|\ln W_{2m}(\tilde{v}_j(\tau_j+1))\right|.
\end{align*}
If $\delta_0=\delta_0(m,n)$ is sufficiently small, 
\begin{align*}
\text{I}\leq C\epsilon_j^{1/2}\left(\int_{I_j}\|\nabla \tilde{v}_j\|_{L^2_\mu}^2+\|\tilde{v}_j\|_{L^2_\mu}^2 d \tau\right)^{\frac{1}{n+1}}W_{2m}(\tilde{v}_j(\tau_j+1))^{1-\frac{1}{n+1}}.
\end{align*}
Similarly, using the monotonicity of $W_{2m}(\tilde{v}(\tau))$, H\"older's inequality and \eqref{eq:contra2m2} we have
\begin{align*}
\text{II}\leq C\epsilon_j^{1/2}\int_{I_j}W_{2m}(\tilde{v}_j(\tau)) d\tau.
\end{align*}
Thus, combining the above estimates and using again the monotonic property of the Weiss energy  we arrive at 
\begin{align*}
&\int\limits_{I_j}\int\limits_{\R^{n-1}\times \{0\}}-\p_n\tilde{u}_j(p_j)_-\ d\mu \leq C\epsilon_j^{1/2}\left(\int\limits_{I_j}\|\nabla \tilde{v}_j\|_{L^2_\mu}^2\ d\tau\right)^{\frac{1}{n+1}}\left(\int\limits_{I_j}W_{2m}(\tilde{v}_j(\tau))\ d\tau\right)^{1-\frac{1}{n+1}}\\
&+C\epsilon_j^{1/2}\left(\int\limits_{I_j}\|\tilde{v}_j\|_{L^2_\mu}^2\ d \tau\right)^{\frac{1}{n+1}}\left(\int\limits_{I_j} W_{2m}(\tilde{v}_j(\tau))\ d\tau\right)^{1-\frac{1}{n+1}} + C\epsilon_j^{1/2}\int\limits_{I_j}W_{2m}(\tilde{v}_j(\tau))\ d\tau.
\end{align*}
Using $W_{2m}(\tilde{v}_j)\leq \frac{1}{4}\|\nabla \tilde{v}_j\|_{L^2_\mu}^2$ to estimate the first and the third term, and applying Young's inequality to the second term, we obtain \eqref{eq:M}.
\end{proof}

Similar as for the case $\kappa=3/2$, Proposition \ref{prop:negative_Weiss} and Proposition \ref{prop:homo2m2} imply a decay estimate for the Weiss energy.

\begin{cor}\label{cor:decay_rate_2m}
Let $\tilde{u}$ be a solution to \eqref{eq:bulk_conf}--\eqref{eq:boundary_conf} with $\kappa=2m$ and satisfy \eqref{eq:sob_conf}.  Then
\begin{itemize}
\item[(i)] Weiss energy goes to $-\infty$ exponentially fast if at the initial time it is negative: there exists $\gamma_{m}\in (0,1)$ such that if $W_{2m}(\tilde{u}(0))<0$, then
\begin{equation}\label{eq:growth_2m}
W_{2m}(\tilde{u}(\tau))\leq e^{\gamma_{m} \tau} W_{2m}(\tilde{u}(0)).
\end{equation}
Moreover, in this case we have
\begin{equation*}
\|\tilde{u}(\tau)\|_{L^2_\mu}^2 \geq -\frac{2W_{2m}(\tilde{u}(0))}{\gamma_m}\left( e^{\gamma_m \tau}-1\right) + \|\tilde{u}(0)\|_{L^2_\mu}^2.
\end{equation*}
\item[(ii)] If $W_{2m}(\tilde{u}(\tau))\geq 0$ for all $\tau\in (0,\infty)$, and at $\tau=0$ one has $\|\tilde{u}(0)\|_{L^2_\mu}=1$ and $W_{2m}(\tilde{u}(0))\leq \delta_0$ for some $\delta_0=\delta_0(m,n)>0$ small, then the Weiss energy satisfies the following decay estimate: there exists $c_0\in (0,1)$ depending on $n,m$ such that 
\begin{equation*}
W_{2m}(\tilde{u}(\tau))\leq \frac{C}{(A_0+c_0\tau)\left|\ln (A_0+c_0\tau)\right|^2},\ \tau\in (0,\infty)
\end{equation*}
where $A_0:=(W_{2m}(\tilde{u}(0))\left|\ln W_{2m}(\tilde{u}(0))\right|^2)^{-1}$.
Moreover, there exists a unique non-zero $\tilde{u}(\infty):=\lim_{\tau\rightarrow \infty}\tilde{u}(\tau)\in \mathcal{E}_{2m}^+$ and a constant $C>0$ depending on $m$ and $n$, such that
\begin{equation*}
0\leq \|\tilde{u}(\infty)\|_{L^2_\mu}^2-\|\tilde{u}(\tau)\|_{L^2_\mu}^2 \leq \frac{C}{\ln\tau}.
\end{equation*}
\end{itemize}
\end{cor}
\begin{proof}
The proof for (i) is the same as for Corollary \ref{cor:decay_rate_32}. We will only provide the proof for (ii).

We first show the decay estimate for the Weiss energy. By the same arguments as above we can get a slightly more general decay estimate for the Weiss energy: under the same assumptions of Proposition \ref{prop:homo2m2}, there exists a $c_0\in (0,1)$ only depending on $m,n$ such that
\begin{equation*}
W_{2m}(\tilde{u}(\tau+h))\leq \left(1-c_0 h W_{2m}(\tilde{u}(\tau+h))\left|\ln W_{2m}(\tilde{u}(\tau+h))\right|^2\right) W_{2m}(\tilde{u}(\tau)),
\end{equation*}
for all $\tau\in (0,\infty)$ and $h\in (0,1]$. This implies that for any $0<\tau_1<\tau_2<\infty$,
\begin{equation}\label{eq:diff_w_2m}
W_{2m}(\tilde{u}(\tau_2))-W_{2m}(\tilde{u}(\tau_1))=\int\limits_{\tau_1}^{\tau_2}\frac{d}{d\tau}W_{2m}(\tilde{u}(\tau))\leq -c_0 \int\limits_{\tau_1}^{\tau_2}W_{2m}(\tilde{u}(\tau))^2\left|\ln W_{2m}(\tilde{u}(\tau))\right|^2.
\end{equation}
To solve this differential inequality we introduce 
\begin{equation*}
F(s)=-\frac{1}{s(\ln s)^2}-2\int_{-\ln s_0}^{-\ln s} \frac{e^u}{u^3}\ du,
\end{equation*}
where $s\in (0,s_0)$ with $0<s_0\ll 1$. Direct computation yields that $F'(s)=\frac{1}{s^2|\ln s|^2}$ and 
\begin{equation}\label{eq:log_asymp}
- \frac{2 \ln|\ln s|}{s|\ln s|^3}\leq -2\int_{-\ln s_0}^{-\ln s} \frac{e^u}{u^3}\ du \leq 0, \ 0<s< s_0,
\end{equation}
which is of $o(\frac{1}{s|\ln s|^2})$ as $s\rightarrow 0$. 
From \eqref{eq:diff_w_2m} we obtain
\begin{equation*}
F(W_{2m}(\tilde{u}(\tau)))-F(W_{2m}(\tilde{u}(0)))\leq -c_0 \tau.
\end{equation*}
Then the expression of $F$ and \eqref{eq:log_asymp} yield
\begin{equation*}
\frac{2}{W_{2m}(\tilde{u}(\tau))\left|\ln W_{2m}(\tilde{u}(\tau))\right|^2}\geq A_0+c_0\tau, \ A_0:=\frac{1}{W_{2m}(\tilde{u}(0))\left|\ln W_{2m}(\tilde{u}(0))\right|^2}.
\end{equation*}
Let $G(w):=w|\ln w|^2$, then the above inequality can be rewritten in terms of $G$ as $G(W_{2m}(\tilde{u}(\tau)))\leq \frac{1}{A_0+c_0\tau}$.  Noticing that $G$ is monotone increasing for $0<w\ll 1$ and $G^{-1}(y)\leq \frac{y}{|\ln y|^2}$, we thus obtain
\begin{equation*}
W_{2m}(\tilde{u}(\tau))\leq \frac{C}{(A_0+c_0\tau)\left|\ln (A_0+c_0\tau)\right|^2}
\end{equation*}
for some universal $C>0$, provided $\delta_0$ is sufficiently small.

The decay of the Weiss energy together with \eqref{eq:w2} yields that $\lim_{\tau\rightarrow \infty} \|\tilde{u}(\tau)\|^2_{L^2_\mu}$ is non-zero if $\delta_0$ is sufficiently small. Because otherwise, integrating \eqref{eq:w2} gives $0\leq \|\tilde{u}(\tau)\|_{L^2_\mu}^2\leq C(\ln(A_0+c_0\tau))^{-1}$ for some $C$ depends on $m,n$. Evaluating at $\tau=0$ thus yields $\|\tilde{u}(0)\|_{L^2_\mu}^2\leq C(\ln A_0)^{-1}\leq C(-\ln W_{2m}(\tilde{u}(0)))^{-1}$. Thus we obtain a contradiction if $\delta_0$ is sufficiently small.

The boundedness of the Weiss energy together with \eqref{eq:Weiss_new2} yields that along a sequence of time $\tau_j\rightarrow \infty$, $\tilde{u}(\tau_j,\cdot)\rightarrow \tilde{u}_0$, where $\tilde{u}_0$ is a non-zero stationary solution to the Signorini problem. Then by the classification of the stationary solution (cf. Proposition \ref{prop:wlimit}) we have that $u_0=\lambda_0 p_0\in \mathcal{E}_{2m}^+$ with $\|p_0\|_{L^2_\mu}=1$ and $\lambda_0>0$. To show $u_0$ is the unique limit, for each $\tau$ we project $\tilde{u}$ to the linear space spanned by $p_0$, i.e. $\tilde{u}(\tau)=\lambda(\tau)p_0+\tilde{w}(\tau)$ with $\langle p_0, \tilde{w}(\tau)\rangle_{L^2_\mu}=0$. Noting that $\tau\mapsto \|\tilde{w}(\tau)\|_{L^2_\mu}^2$ is monotone decreasing and $\|\tilde{w}(\tau_j)\|_{L^2_\mu}^2\rightarrow 0$,  we have that $u_0=\lambda_0p_0$ is the unique limit.
\end{proof}

\section{Consequences of the epiperimetric inequality}
\label{sec:consequences}

\subsection[Consequences: $\kappa=3/2$]{The case $\kappa = 3/2$: Uniqueness of blow-ups and regularity of the regular free boundary}
\label{sec:kappa32}
In this section we apply the decay estimates for the Weiss energy in Section \ref{sec:dynamical} to our original Signorini problem to derive the regularity of the free boundary.

\begin{proof}[Proof for Theorem \ref{thm:32}]
We will prove the following:
\begin{align}\label{eq:rescale_rate}
\left(\int\limits_{\R^n_+} |u_{\lambda}(x,t)- u_{0}(x)|^2G(x,t)dx\right)^{1/2} \leq C (\sqrt{-t})^{\kappa+\gamma_0}\lambda^{\gamma_0} ,
\end{align}
where for any $\lambda\in (0,1]$, $u_{\lambda}(x,t):= \lambda^{-3/2}u(\lambda x, \lambda^2 t)$. Theorem \ref{thm:32} follows by taking $\lambda=1$.

Let $\tilde{u}:=\tilde{u}_{3/2}$ be the $3/2$-normalized solution in the conformal coordinates as in Lemma \ref{lem:coordinates}.  Firstly we want to show that there exists $\gamma_0\in (0,1)$ universal, such that the upper bound in \eqref{eq:vanishing_32} yields $W_{3/2}(\tilde{u}(\tau))\geq 0$ for all $\tau$. Indeed, assume that $W_{3/2}(\tilde{u}(\tau_0))<0$ for some $\tau_0>0$. Then by Corollary \ref{cor:decay_rate_32}, there exists a universal $\gamma_0\in (0,1)$ and $C_0>0$, such that $\|\tilde{u}(\tau)\|_{L^2_\mu}^2 \geq C_0e^{\gamma_0(\tau-\tau_0)}$ for $\tau>\tau_0$. Transforming back to the original coordinate we have $\int_{\R^n_+} u^2(x,t)G(x,t)dx\geq C(\sqrt{-t})^{3-2\gamma_0}$ for some $C>0$ depending on $\tau_0$ and for $|t|$ sufficiently small. This is however a contradiction to the upper bound in \eqref{eq:vanishing_32}.

With nonnegative Weiss energy at hand we apply (i) in Corollary \ref{cor:decay_rate_32} to conclude that there exists a unique $\tilde{u}(\infty)=:u_0\in \mathcal{E}_{3/2}$ such that
\begin{equation*}
\begin{split}
\int\limits_{\R^n_+}(\sqrt{-t})^{-3}|u_\lambda(x,t)-u_0(x)|^2G(x,t)dx &= c_n\int\limits_{\R^n_+} |\tilde{u}(y,\tau-2\ln \lambda)-u_0(y)|^2 d\mu \\
&\leq Ce^{-\gamma_0(\tau-2\ln\lambda)}= C(\sqrt{-t})^{2\gamma_0} \lambda^{2\gamma_0}.
\end{split}
\end{equation*}
Here in the first equation we have used \eqref{eq:kappa} and that $u_0$ is $3/2$-homogeneous.

In the end we will show that the lower bound in \eqref{eq:vanishing_32} implies that $u_0\neq 0$. Indeed, if $u_0$ vanishes identically, then it holds
\begin{align*}
\int\limits_{\R^n_+} |u (x,t)|^2 G(x,t) dx \leq C (\sqrt{-t})^{3+2\gamma_0}, \quad t\in [-1,0].
\end{align*}
The above estimate yields that $H_u(r)\leq C r^{3+2\gamma_0}$. This is a contradiction to the lower bound in \eqref{eq:vanishing_32}.
\end{proof}

\begin{rmk}\label{rmk:non_zero2}
Rewriting \eqref{eq:smallness_32} in Remark \ref{rmk:non_zero} into the original variable, we see that instead of the lower bound assumption in \eqref{eq:vanishing_32}, we can assume the solution is close to $\mathcal{E}_{3/2}$ at $t=-1$ to guarantee the non-triviality of the $3/2$-blowup limit. More precisely,
assume that at $t=-1$
\begin{equation}\label{eq:close1}
W_{3/2}(u(-1)) \leq \delta_0\int_{\R^n_+}|u(x,-1)|^2 d\tilde{\mu}(x), \quad d\tilde{\mu}(x):=G(x,-1)dx,
\end{equation}
where $W_{3/2}(u_0)$ is the Weiss energy in the original variable as in \eqref{eq:weiss_original}, and
\begin{equation}\label{eq:close2}
\inf_{p\in \mathcal{E}_{3/2}}\int_{\R^n_+} |u(x,-1)-p(x)|^2 d\tilde{\mu}(x) \leq \delta_0\int_{\R^n_+}|u(x,-1)|^2 d\tilde{\mu}(x).
\end{equation}
Then if $\delta_0$ is sufficiently small depending on $n$ and $\|f\|_{L^\infty}$, there is a unique $u_0=c_0 \Ree(x'\cdot e_0+i|x_n|)^{3/2}\in \mathcal{E}_{3/2}$ with $c_0\geq c_n>0$ such that \eqref{eq:rescale_rate} holds true. We note that conditions \eqref{eq:close1}--\eqref{eq:close2} are satisfied if
$$\dist_{W^{1,2}_{\tilde{\mu}}}\left(\frac{u(\cdot, -1)}{\|u(\cdot, -1)\|_{L^2_\mu}}, \mathcal{E}_{3/2}\right)\leq \delta_0.$$

We also note that under the assumptions of Theorem \ref{thm:32},  \eqref{eq:close1}--\eqref{eq:close2} are satisfied for $u_\lambda$ for sufficiently small $\lambda>0$ depending on $u_0$.
\end{rmk}

An advantage of the conditions \eqref{eq:close1}--\eqref{eq:close2} is that they are stable under the translation. More precisely, by the H\"older continuity of $u$, \eqref{eq:close1}--\eqref{eq:close2} hold with constant $2\delta_n$ for $u(x-x_0,t-t_0)$, where $(x_0,t_0)$ varies in a small neighborhood of $(0,0)$. We let $\Gamma_{3/2}(u)$ denote the set of the free boundary points at which the $3/2$-homogeneous scaling $u_{(x_0,t_0),\lambda}$ has a unique \emph{nonzero} blow-up limit in $\mathcal{E}_{3/2}$ as $\lambda\rightarrow 0$. Then the above discussion leads to the openness of $\Gamma_{3/2}(u)$:

\begin{prop}\label{prop:stability}
Let $u:S_2^+\rightarrow \R$ be a solution to \eqref{eq:parabolic} which satisfies assumptions (A)-(C). Let $H_u^{(x_0,t_0)}(r):=\frac{1}{r^2}\int_{S_{r^2}^+}u(x-x_0,t-t_0)^2 G(x,t)dxdt$, $r>0$. Assume that for each $(x_0,t_0)\in \Gamma_u$,
\begin{equation}\label{eq:vanishing_322}
H_u^{(x_0,t_0)}(r)\leq Cr^{3-\gamma_0}
\end{equation}
for $r$ sufficiently small depending on $u$ and $(x_0,t_0)$. Assume that $(\bar x_0,\bar t_0)\in \Gamma_{3/2}(u)\cap S'_1$. Then there exists a small $r>0$ depending on $(\bar x_0,\bar t_0)$ such that $\Gamma_u\cap (B_r(\bar x_0)\times (\bar t_0-r^2,\bar t_0+r^2))\subset \Gamma_{3/2}(u)$.
\end{prop}

Given $(x_0,t_0)\in \Gamma_{3/2}(u)$, we let
$$u_{(x_0,t_0)}(x)=c_{(x_0,t_0)}\Ree(e_{(x_0,t_0)}\cdot x + i|x_n|)^{3/2},\quad c_{(x_0,t_0)}>0$$
denote the blow-up limit. Next we prove the continuity of the maps $\Gamma_{3/2}(u)\ni (x_0,t_0) \mapsto c_{(x_0,t_0)}$ and $\Gamma_{3/2}(u) \ni (x_0,t_0)\mapsto e_{(x_0,t_0)}$.

\begin{prop}
\label{prop:coeff}
Let $u:S_1^+ \rightarrow \R$ be a solution to \eqref{eq:parabolic}. 
Let $(x_0,t_0), (y_0,s_0)\in \Gamma_{3/2}(u) \cap Q_1$, where $Q_1:=B_1(0)\times (-1,0)$. Then, the maps
\begin{align*}
\Gamma_{3/2}(u)\ni (x_0,t_0) \mapsto c_{(x_0,t_0)} \in \R_+,\quad
\Gamma_{3/2}(u) \ni (x_0,t_0)\mapsto e_{(x_0,t_0)} \in \mathbb{S}^{n-1}\cap\{y_n=0\}
\end{align*}
are parabolically $\theta$-H\"older continuous for some $\theta \in (0,1)$.
\end{prop}

\begin{proof}
We note that by rotation invariance $c_{(x_0,t_0)}= c_n\|u_{(x_0,t_0)}\|_{L^2_{\tilde{\mu}}}$ for $c_n>0$. Hence,  for $(x_0,t_0), (y_0,s_0)\in \Gamma_{3/2}(u)\cap Q_1$ and $\lambda>0$
\begin{align*}
&\quad |c_{(x_0,t_0)}-c_{(y_0,s_0)}|\leq c_n \left|\|u_{(x_0,t_0)}\|_{L^2_{\tilde{\mu}}}-\|u_{(y_0,s_0)}\|_{L^2_{\tilde{\mu}}}\right|\\
&\leq c_n \left(\|u_{(x_0,t_0)}-u_{(x_0,t_0),\lambda}(\cdot, -1)\|_{L^2_{\tilde{\mu}}}
+\|u_{(y_0,s_0)}-u_{(y_0,s_0),\lambda}(\cdot, -1)\|_{L^2_{\tilde{\mu}}}\right.\\
&\quad \left.+\left|\|u_{(x_0,t_0),\lambda}(\cdot, -1)\|_{L^2_{\tilde{\mu}}}-\|u_{(y_0,s_0),\lambda}(\cdot,-1)\|_{L^2_{\tilde{\mu}}}\right|\right)\\
&\leq C\lambda^{\gamma_0} + Cd((x_0,t_0), (y_0,s_0))^\alpha \lambda^{-3/2}.
\end{align*}
Here $d((x_0,t_0), (y_0,s_0))=|x_0-y_0|+|t_0-s_0|^{1/2}$ is the parabolic distance, and to estimate the three terms coming from the triangle inequality, we have used \eqref{eq:rescale_rate} to bound the first two integrals and the interior H\"older $C^{\alpha,\alpha/2}$ estimate of the solution to bound the third integral. Balancing the above two bounds we get
\begin{equation*}
|c_{(x_0,t_0)}-c_{(y_0,s_0)}|\leq Cd((x_0,t_0),(y_0,s_0))^{\theta}, \quad \theta=\frac{\gamma_0 \alpha}{\gamma_0+3/2}\in (0,1).
\end{equation*}
Next we note that
\begin{align*}
\left\| u_{(x_0,t_0)}/c_{(x_0,t_0)} - u_{(y_0,s_0)}/c_{(y_0,s_0)} \right\|_{L^2_{\tilde{\mu}}} \geq C_n |e_{(x_0,t_0)}-e_{(y_0,s_0)}|.
\end{align*}
Using similar estimate as above and combining it with the estimate for $c_{(x_0,t_0)}$ then yields the claimed H\"older continuity of $e_{(x_0,t_0)}$.
\end{proof}

With the previous results at hand, we can prove the regularity of the regular free boundary.

\begin{prop}
\label{prop:free_bound}
Let $u:S_1^+ \rightarrow \R$ be a solution to \eqref{eq:parabolic}. Assume that $(x_0,t_0)\in \Gamma_{3/2}(u)$. Then, there exists a radius $r\in (0,1)$ depending on $(x_0,t_0)$ such that $\Gamma_u\cap Q_{r}(x_0,t_0)$, where $Q_r(x_0,t_0)=B_r(x_0)\times (t_0-r^2,t_0+r^2)$, can be represented as a graph (after a suitable choice of coordinates)
\begin{align*}
\Gamma_{3/2}(u)\cap Q_{r}(x_0,t_0):=\{(x',0,t): x_{n-1} = g(x'',t)\}.
\end{align*}
Moreover, there exists $\theta\in(0,1)$ such that $\nabla'' g \in C^{\theta,\theta/2}$.
\end{prop}

\begin{proof}
From Proposition \ref{prop:stability} we know that there exists $r>0$ depending on $(x_0,t_0)$ such that $\Gamma_u\cap Q_r(x_0,t_0)$ consist of $\Gamma_{3/2}(u)$ free boundary points. Let $c_{(x_0,t_0),\lambda}>0$ and $e_{(x_0,t_0),\lambda}$ be the direction such that $c_{(x_0,t_0),\lambda}\Ree(x'\cdot e_{(x_0,t_0),\lambda}+i|x_n|)^{3/2}\in \mathcal{E}_{3/2}$ realizes the $L^2_{\tilde{\mu}}$ distance between $u_\lambda(\cdot, -1)$ and $\mathcal{E}_{3/2}$. Then from the proof for Proposition \ref{prop:coeff} we have that for each $\lambda$ sufficiently small,
\begin{equation*}
|e_{(x_0,t_0),\lambda}-e_{(y_0,s_0),\lambda}|\leq C d((x_0,t_0),(y_0,s_0))^\theta.
\end{equation*}
for any $(x_0,t_0), (y_0,s_0)\in \Gamma_u\cap Q_r(x_0,t_0)$, and $C>0$ independent of $\lambda$. Thus we find a parameter family of hypersurfaces $\Gamma_u^\lambda$, the normals of which are spacial and equal to $e_{(x_0,t_0),\lambda}$ at each $(x_0,t_0)\in \Gamma_u^\lambda$, and they are uniformly $C^\theta$ regular with respect to the parabolic distance. Passing to the limit as $\lambda\rightarrow 0$ we thus obtain that the limiting hypersurface, which is the free boundary $\Gamma_u\cap Q_r(x_0,t_0)$, is a $C^\theta$ hypersurface. Thus up to a rotation of the spacial cooridnates, it can be represented as the graph $x_{n-1}=g(x'',t)$ for some function $g$, where $\nabla '' g\in C^{\theta,\theta/2}$.
\end{proof}

\subsection{The case $\kappa=2m$.}
\subsubsection{Uniqueness and nondegeneracy}
We first prove Theorem \ref{thm:2m} by using decay estimate of the Weiss energy in Corollary \ref{cor:decay_rate_2m}.
\begin{proof}[Proof for Theorem \ref{thm:2m}]
Let $\tilde{u}=\tilde{u}_{2m}$ be the $2m$ conformal normalized solution as in Lemma \ref{lem:coordinates}.  The upper bound on $H_u(r)$ implies that $W_{2m}(\tilde{u}(\tau))\geq 0$ for all $\tau\in (0,\infty)$. In fact, if $W_{2m}(\tilde{u}(\tau_0))<0$ for some $\tau_0>0$, then by Corollary \ref{cor:decay_rate_2m}, there exist $\gamma_m\in (0,1)$ and $C_0>0$ such that $\|\tilde{u}(\tau)\|_{L^2_\mu}^2 \geq C_0e^{\gamma_m (\tau-\tau_0)}$ for each $\tau>\tau_0$. Back in the original coordinates we have $\int_{\R^n_+} u^2(x,t)G(x,t)dx\geq C(\sqrt{-t})^{4m-2\gamma_m}$ for some $C>0$ depending on $\tau_0$ and for each $t\in (-e^{-\tau_0}, 0)$. This is however a contradiction to our assumption on $H_u$ when $|t|$ is sufficiently small.
Thus one can apply Corollary \ref{cor:decay_rate_2m} (ii) to $\tilde{u}(\tau)/\|\tilde{u}(0)\|_{L^2_\mu}$ and conclude that there is a unique non-zero $p(y)\in \mathcal{E}_{2m}^+$ such that
$$\|\tilde{u}(\tau)-p\|_{L^2_\mu}\rightarrow 0, \quad 0\leq \|\tilde{u}(\tau)\|^2_{L^2_\mu}-\|p\|_{L^2_\mu}^2\leq \frac{C\|\tilde{u}(0)\|_{L^2_\mu}^2}{|\ln\tau|} \text{ as } \tau\rightarrow \infty.$$
Let $p_0(x,t):=(\sqrt{-t})^\kappa p(\frac{x}{2\sqrt{-t}})$. Writing the above inequality by the original variables and by \eqref{eq:kappa}, we obtain the desired estimate.
\end{proof}

\subsubsection{Frequency gap}
By the Almgren-Poon's monotonicity formula for solutions to the parabolic Signorini problem (with $f=0$), each free boundary point $(x_0,t_0)\in \Gamma_u$ can be associated with a frequency $\kappa_{(x_0,t_0)}\in \{3/2\}\cup [2,\infty)$. Furthermore, blow-ups at $(x_0,t_0)$ are parabolic $\kappa_{(x_0,t_0)}$-homogeneous solutions, cf. \cite{DGPT}. Another consequence of the epiperimetric inequality is the gap of the frequency around $2m$, $m\in \N_+$.

\begin{prop}[Frequency gap]
Let $u: S_2^+\rightarrow \R$ be a solution to the parabolic Signorini problem \eqref{eq:parabolic} with $f=0$. Assume that $u$ satisfies assumptions (A)-(C). Then there exist positive constants $c_-$ and $c_+$ depending on $m,n$ such that
\begin{equation*}
\{(x,t)\in \Gamma_u: \kappa_{(x,t)}\in (2m-c_-,2m+c_+)\setminus\{2m\}\}=\emptyset.
\end{equation*}
\end{prop}
\begin{proof}
\emph{(i).} Assume that $(0,0)\in \Gamma_u$ is a free boundary point with the frequency $2m+\epsilon$ for some $\epsilon>0$. Let $u(x,t)$ be a non-trivial $2m+\epsilon$ parabolically homogeneous blow-up limit at $(0,0)$. In the conformal coordinate it is of the form $e^{-(m+\epsilon/2)\tau} v(y)$ for some non-trivial function $v$, which solves the stationary Signorini problem
\begin{equation*}
\begin{split}
\mathcal{L}_{2m}v + \frac{\epsilon}{2}v &=0\text{ in }\R^{n}_+,\\
v\geq 0, \ \p_n v\leq 0, \ v\p_nv & =0 \text{ on } \{y_n=0\}.
\end{split}
\end{equation*}
 We consider the $2m$-normalized solution $\tilde{u}_{2m}(y,\tau)= e^{m\tau}u(x(y,\tau), t(y,\tau))=e^{-\epsilon \tau/2}v(y)$ as in Lemma \ref{lem:coordinates}. By Proposition \ref{prop:homo2m2},
\begin{equation*}
W_{2m}(\tilde{u}_{2m}(\tau+1))\leq \left(1-c_0 W_{2m}(\tilde{u}_{2m}(\tau+1))\left|\ln W_{2m}(\tilde{u}_{2m}(\tau+1))\right|^2\right)W_{2m}(\tilde{u}_{2m}(\tau)).
\end{equation*}
Note that $W_{2m}(\tilde{u}(\tau))=e^{-\epsilon\tau}W_{2m}(v)$ and $W_{2m}(v)=\epsilon\|v\|_{L^2_\mu}^2/2$. Thus the above inequality can be rewritten as (after dividing by $W_{2m}(v)$)
\begin{equation*}
e^{-\epsilon(\tau+1)}\leq \left(1-c_0 e^{-\epsilon(\tau+1)} (\epsilon\|v\|_{L^2_\mu}^2/2)\left|-\epsilon(\tau+1)+\ln(\epsilon\|v\|_{L^2_\mu}^2/2)\right|^2\right)e^{-\epsilon\tau}.
\end{equation*}
Evaluating at $\tau=0$ yields
\begin{equation*}
e^{-\epsilon}\leq 1- \tilde{c} e^{-\epsilon} \epsilon \left|-\epsilon+\ln \epsilon + \ln (\tilde{c}/c_0)\right|^2 ,\quad \tilde{c}:=\frac{c_0\|v\|^{2}_{L^2_\mu}}{2}.
\end{equation*}
Then necessarily $\epsilon\geq \epsilon_0$ (thus $c_+\geq \epsilon_0$) for some $\epsilon_0>0$ depending on $\tilde{c}$.

\emph{(ii).} Assume that $(0,0)$ is a free boundary point with frequency $2m-\epsilon$, $\epsilon>0$. Then $\epsilon>\gamma_m$, where $\gamma_m\in(0,1)$ is the constant from Corollary \ref{cor:decay_rate_2m} (i). Indeed, similar as in (i) we consider a nontrivial blow-up limit at $(0,0)$. After  $2m$-normalization and in the conformal coordinates this leads to $\tilde{u}_{2m}(y,\tau)=e^{\epsilon\tau/2} v(y)$, where $v$ solves
\begin{equation*}
\begin{split}
\mathcal{L}_{2m}v-\frac{\epsilon}{2}v &=0 \text{ in }\R^{n}_+,\\
v\geq 0, \ \p_n v\leq 0, \ v\p_nv & =0 \text{ on } \{y_n=0\}.
\end{split}
\end{equation*}
Thus one has
\begin{equation*}
W_{2m}(\tilde{u}_{2m}(\tau))=-\frac{\epsilon}{2}e^{\epsilon\tau}\|v\|_{L^2_\mu}^2<0,\quad \tau\in (0,\infty).
\end{equation*}
By \eqref{eq:growth_2m}, $W_{2m}(\tilde{u}_{2m}(\tau))\leq e^{\gamma_m\tau}(-\frac{\epsilon}{2}\|v\|_{L^2_\mu})$, which implies that  $\epsilon\geq \gamma_m>0$.
\end{proof}

\section{Perturbation}\label{sec:perturbation}
In this section we show how to modify our proof in Section \ref{sec:dynamical} to the nonzero inhomogeneity setting. We consider global solutions which satisfies (A)-(C) to
\begin{equation*}
\begin{split}
\p_tu-\Delta u &=f \text{ in } S_2^+\\
 u \geq 0, \ \p_nu\leq 0, \ u\p_nu &=0 \text{ on } S'_2,
\end{split}
\end{equation*}
where $f=f(x,t)\in L^\infty(S_2^+)$. By chapter 4 of \cite{DGPT}, the study of local solutions with nonzero obstacles can be reduced to the study of global solutions to the above inhomogeneous equations by subtracting the obstacle and applying suitable cut-offs.

Assume $(0,0)\in \Gamma_u$ is a free boundary point of frequency $\kappa$. Under a similar conformal change of variables around $(0,0)$ as before (cf. Lemma \ref{lem:coordinates}) we have that $\tilde{u}_\kappa$ solves
\begin{equation}\label{eq:inhomogeneous}
\begin{split}
&\p_\tau \tilde{u}_\kappa +\frac{y}{2}\cdot \nabla \tilde{u}_\kappa-\frac{1}{4}\Delta \tilde{u}_\kappa-\frac{\kappa}{2}\tilde{u}_\kappa= e^{\tau(\kappa/2-1)} \tilde{f} \text{ in } \R^n_+\times [0,\infty),\\
&\tilde{u}_\kappa\geq 0, \ \p_n\tilde{u}_\kappa\leq 0, \ \tilde{u}_\kappa \p_n\tilde{u}_\kappa=0 \text{ on } \{y_n=0\},
\end{split}
\end{equation}
where $\tilde{f}(y,\tau):=f(2e^{-\tau/2} y, -e^{-\tau})$.
Note that up to a dimensional constant $$M_{\tilde{f}}:=\|\tilde{f}\|_{L^\infty([0,\infty); L^2_\mu)}=\sup_{t\in [-1,0)}\left(\int_{\R^n_+}|f(x,t)|^2 G(x,t) dx\right)^{1/2}\leq \|f\|_{L^\infty(S_1^+)}.$$

\vspace{5pt}

We first consider the case $\kappa=3/2$ and denote $\tilde{u}:=\tilde{u}_{3/2}$. Let $W(\tilde{u}(\tau)):=W_{3/2}(\tilde{u}(\tau))$ be the Weiss energy (defined as in \eqref{eq:Weiss_new}) associated with the solution $\tilde{u}$. A direct computation as in Lemma \ref{lem:Weiss_conf} gives that $\tau\mapsto W(\tilde{u}(\tau))$ satisfies the almost monotone decreasing property: for any $0<\tau_a<\tau_b<\infty$
\begin{align}
\label{eq:der_a}
\begin{split}
W(\tilde{u}(\tau_b))-W(\tilde{u}(\tau_a))
&\leq -2\int\limits_{\R^n_+\times [\tau_a, \tau_b]} |\p_\tau \tilde{u}|^2 d\mu d\tau+ 2 \int\limits_{\R^n_+\times [\tau_a, \tau_b]} e^{-\tau/4} \p_\tau\tilde{u} \tilde{f} d\mu d\tau \\
&\leq -  \int\limits_{\R^n_+\times [\tau_a,\tau_b]}|\p_{\tau} \tilde{u}|^2 d \mu d\tau +\int\limits_{\R^n_+\times [\tau_a, \tau_b]}  e^{-\tau / 2} \tilde{f}^2 d\mu d\tau.
\end{split}
\end{align}

Relying on this almost monotonicity, we seek to prove the following contraction result:

\begin{prop}\label{prop:perturb}
Let $\tilde{u}$ be a solution to \eqref{eq:inhomogeneous} with $\kappa=3/2$ and $\tilde{u}$ satisfies \eqref{eq:sob_conf}.
Then there exists a universal constant $c_0\in (0,1/4)$ such that
\begin{equation*}
W(\tilde{u}(\tau+ 1))\leq (1-c_0)W(\tilde{u}(\tau)) + 2 e^{-\tau/2}M_{\tilde{f}}^2 \text{ for any }\tau>0.
\end{equation*}
\end{prop}

The argument for this contraction is similar as before. Instead of providing the full details, we only outline the proof and point out the differences.

\begin{proof}[Proof of Proposition \ref{prop:perturb}]
\emph{(i).} Assume not, then there exists $c_j\in (0,1/4)$, $c_j\rightarrow 0$, solutions $\tilde{u}_j$ to \eqref{eq:inhomogeneous} with inhomogeneity $\tilde{f}_j$
and times $\tau_j$ such that
\begin{equation*}
W(\tilde{u}_j(\tau_j+1))\geq (1-c_j)W(\tilde{u}_j(\tau_j)) + 2 e^{-\tau_j/2} M_{\tilde f_j}^2 .
\end{equation*}
In the sequel for notational simplicity we write $\tilde{u}$ and $\tilde{f}$ instead of $\tilde{u}_j$ and $\tilde{f}_j$.
Then, using \eqref{eq:der_a}, one has
\begin{align*}
c_jW(\tilde{u}(\tau_j+1))&\geq (1-c_j)\left[W(\tilde{u}(\tau_j))- W(\tilde{u}(\tau_j+1))\right]+2e^{-\tau_j/2}M_{\tilde{f}}^2\\
&\geq (1-c_j) \int\limits_{\tau_j}^{\tau_j +1}\left( \|\p_{\tau} \tilde{u}\|_{L^2_{\mu}}^2
-e^{-\tau/2} \|\tilde{f}\|_{L^2_{\mu}}^2\right) d\tau +2e^{-\tau_j/2}M_{\tilde{f}}^2\\
&\geq (1-c_j)\int\limits_{\tau_j}^{\tau_j +1} \|\p_{\tau} \tilde{u}\|_{L^2_{\mu}}^2 d\tau+ e^{-\tau_j/2}M_{\tilde{f}}^2.
\end{align*}
By the almost monotonicity \eqref{eq:der_a}, for any $\tau\in I_j:=[\tau_j,\tau_j+1]$,
\begin{equation*}
c_jW(\tilde{u}(\tau))\geq c_jW(\tilde{u}(\tau_j+1))-c_je^{-\tau_j/2}M_{\tilde{f}}^2\geq (1-c_j)\left[\int_{I_j} \|\p_{\tau} \tilde{u}\|_{L^2_{\mu}}^2 d\tau+ e^{-\tau_j/2}M_{\tilde{f}}^2\right].
\end{equation*}
Therefore, we obtain
\begin{equation}\label{eq:contra_a}
\int_{I_j} \|\p_{\tau} \tilde{u}\|_{L^2_{\mu}}^2 d\tau+ e^{-\tau_j/2}M_{\tilde{f}}^2\leq 2c_j\int_{I_j}W(\tilde{u}(\tau))d\tau.
\end{equation}

\emph{(ii).} We seek to estimate the Weiss energy for the error term $\tilde{v}_j$. We remark that the relation \eqref{eq:weissv} still holds for the inhomogeneous problem. First, we can write
\begin{equation*}
\begin{split}
&\int_{I_j}W(\tilde{u}(\tau))d\tau=\int\limits_{\R^n_+\times I_j} \mathcal{L}\tilde{u}\tilde{u} d\mu d\tau\\
&=\int\limits_{\R^n_+\times I_j} (-\p_\tau \tilde{u}+ e^{-\tau/4}\tilde{f})\tilde{v} d\mu d\tau +\int\limits_{I_j} \frac{\lambda_j}{2}\int\limits_{\{y_n=0\}}\left(\p_n\tilde{u} h_{e_j}-\tilde{u}\p_nh_{e_j}\right)d\mu  d\tau\\
&\leq\int_{I_j}\left(\|\p_\tau\tilde{u}\|_{L^2_{\mu}}+e^{-\tau/4}\|\tilde{f}\|_{L^2_\mu}\right)\|\tilde{v}\|_{L^2_\mu}d\tau +\int\limits_{I_j}\frac{\lambda_j}{2}\int\limits_{\{y_n=0\}}\left(\p_n\tilde{u} h_{e_j}-\tilde{u}\p_nh_{e_j}\right)d\mu d\tau.
\end{split}
\end{equation*}
Here as in the zero homogeneity case, $\mathcal{L}:=\frac{1}{4}\Delta-\frac{y}{2}\cdot \nabla +\frac{3}{4}$.
Applying \eqref{eq:contra_a} to the first term on the RHS we obtain
\begin{equation*}
\begin{split}
\int_{I_j}W(\tilde{u}(\tau))d\tau &\leq \left(2c_j\int_{I_j}W(\tilde{u}(\tau)d\tau \right)^{1/2}\left(\int_{I_j}\|\tilde{v}\|_{L^2_\mu}^2 d\tau \right)^{1/2}\\
&\quad +\int\limits_{I_j}\frac{\lambda_j}{2}\int\limits_{\{y_n=0\}}\left(\p_n\tilde{u} h_{e_j}-\tilde{u}\p_nh_{e_j}\right)d\mu d\tau.
\end{split}
\end{equation*}
By Cauchy-Schwartz and applying \eqref{eq:weissv} to replace $W(\tilde{u})$ by $W(\tilde{v})$ we get,
\begin{equation*}
\begin{split}
\int_{I_j}W(\tilde{v}(\tau)) d\tau \leq 20 c_j\int_{I_j}\|\tilde{v}\|^2_{L^2_\mu} d\tau + 4\int_{I_j}\int_{\{y_n=0\}}\left(\frac{\lambda_j}{4}\p_n\tilde{u} h_{e_j}+ \frac{\lambda_j}{8} \tilde{u}\p_nh_{e_j}\right)d\mu d\tau .
\end{split}
\end{equation*}
Noting that the term involving integral on $\{y_n=0\}$ is non-positive by the Signorini boundary condition, we thus obtain
\begin{equation}\label{eq:estiv}
\begin{split}
\int_{I_j}W(\tilde{v}(\tau)) d\tau \leq C c_j\int_{I_j}\|\tilde{v}\|^2_{L^2_\mu} d\tau.
\end{split}
\end{equation}
Using $W(\tilde{v}(\tau))\geq -\frac{3}{4}\|\tilde{v}(\tau)\|_{L^2_\mu}^2$ we further obtain
\begin{equation}
\label{eq:estiv2}
\begin{split}
-\frac{\lambda_j}{8}\int_{I_j}\int_{\{y_n=0\}}\left(\p_n\tilde{u} h_{e_j}+\tilde{u}\p_nh_{e_j}\right)d\mu d\tau\leq C\int_{I_j}\|\tilde{v}\|^2_{L^2_\mu} d\tau.
\end{split}
\end{equation}
Here in \eqref{eq:estiv} and \eqref{eq:estiv2}, the constant $C>0$ is an absolute constant.

With \eqref{eq:contra_a}, \eqref{eq:estiv} and \eqref{eq:estiv2} at hand, we argue as step (iii) and (iv) in Proposition \ref{prop:epi} and reach a contradiction.
\end{proof}

Similar as for the zero inhomogeneity case, if the Weiss energy is negative starting from some time $\tau_0$, one can show a stronger decay estimate:

\begin{prop}\label{prop:perturb_neg}
Let $\tilde{u}$ be a solution to \eqref{eq:inhomogeneous} with $\kappa=3/2$ and $\tilde{u}$ satisfies \eqref{eq:sob_conf}.
Assume that $W(\tilde{u}(\tau))\leq -4e^{-\tau/2}M_{\tilde{f}}^2$ for all $\tau>0$.
Then there exists a universal constant $c_0\in (0,1/4)$ such that
\begin{equation*}
W(\tilde{u}(\tau+ 1))\leq (1+c_0)W(\tilde{u}(\tau)) + 2 e^{-\tau/2}M_{\tilde{f}}^2 \text{ for any }\tau>0.
\end{equation*}
\end{prop}

Now we outline how to obtain the decay rate of the Weiss energy and the convergence rate of $\tilde{u}(\tau)$ from Proposition \ref{prop:perturb} and Proposition \ref{prop:perturb_neg}. In the perturbation case, we consider the modified energy
$$\widetilde{W}(\tilde{u}(\tau)):=W(\tilde{u}(\tau))+ 4 e^{-\frac{\tau}{2}} M_{\tilde{f}}^2.$$
By \eqref{eq:der_a}, $\tau\mapsto \widetilde{W}(\tilde{u}(\tau))$ is monotone decreasing. Moreover, from Proposition \ref{prop:perturb} and Proposition \ref{prop:perturb_neg}, there exists a $c_0\in (0,1/8)$ ($c_0$ might be different from that in Proposition \ref{prop:perturb} but remains universal) such that
\begin{equation*}
\widetilde{W}(\tilde{u}(\tau+1))\leq (1-c_0) \widetilde{W}(\tilde{u}(\tau)),
\end{equation*}
and if $\widetilde{W}(\tilde{u}(\tau))<0$ for all $\tau>0$,
\begin{equation*}
\widetilde{W}(\tilde{u}(\tau+1))\leq (1+c_0) \widetilde{W}(\tilde{u}(\tau)).
\end{equation*}
This implies that there is some constant $\gamma_0\in (0,1)$ depending on $n$ such that
\begin{equation*}
\widetilde{W}(\tilde{u}(\tau))\leq e^{-\gamma_0\tau}\widetilde{W}(\tilde{u}(0)),
\end{equation*}
and
$$\widetilde{W}(\tilde{u}(\tau))\leq e^{\gamma_0\tau} \widetilde{W}(\tilde{u}(0))\text{ if }\widetilde{W}(\tilde{u}(0))<0.$$
From this and arguing similarly as in Corollary \ref{cor:decay_rate_32} we conclude that if $\widetilde{W}(\tilde{u}(\tau))\geq 0$ for all $\tau>0$, then for all $0<\tau_1<\tau_2<\infty$,
\begin{equation*}
\|\tilde{u}(\tau_2)-\tilde{u}(\tau_1)\|_{L^2_\mu}^2\leq  C_n\left(W(\tilde{u}(0))+ 4M_{\tilde{f}}^2\right) e^{-\gamma_0\tau_1};
\end{equation*}
and if $\widetilde{W}(\tilde{u}(\tau_0))<0$ for some $\tau_0>0$, then
\begin{equation*}
\|\tilde{u}(\tau)\|_{L^2_\mu}^2 \geq -C_n \widetilde{W}(\tilde{u}(\tau_0))e^{\gamma_0(\tau-\tau_0)}, \quad \tau\in [\tau_0,\infty).
\end{equation*}

\vspace{5pt}

For the case $\kappa=2m$, $m\in \N_+$ in \eqref{eq:inhomogeneous} we assume further that the inhomogeneity in the conformal coordinates satisfies: for some $M>0$ and $\epsilon_0\in (0,1)$,
\begin{equation}\label{eq:inhomo_2m}
\|\tilde{f}(\tau)\|_{L^2_\mu} \leq M e^{-\tau(m-1+\epsilon_0)}, \text{ for all }\tau>0.
\end{equation}
Note that \eqref{eq:inhomo_2m} is satisfied if in the original coordinates $f(x,t)$ has the vanishing property at $(0,0)$ that $|f(x,t)|\leq M (|x|+\sqrt{-t})^{2(m-1+\epsilon_0)}$.  Under such assumption, the inhomogeneity only contributes as a higher order term in our estimates. In particular, similar as \eqref{eq:der_a} in the $\kappa=3/2$ case, we have the almost monotone decreasing property for the Weiss energy: for $0<\tau_a<\tau_b<\infty$,
\begin{equation*}
W(\tilde{u}(\tau_b))-W(\tilde{u}(\tau_a))\leq -\int\limits_{\R^n_+\times [\tau_a,\tau_b]} |\p_\tau\tilde{u}|^2 d\mu d\tau + M^2 \int\limits_{\tau_a}^{\tau_b} e^{-2\tau \epsilon_0} d\tau.
\end{equation*}
Thus with slight modification as in the case $\kappa=3/2$, one can generalize Proposition \ref{prop:negative_Weiss} and Proposition \ref{prop:homo2m2} to the nonzero inhomogeneity case, and we do not repeat the proof here.

\appendix
\section*{Appendix}
\label{sec:appendix}
In the appendix we prove the following characterization result for the second Dirichlet-Neumann eigenspace for the Ornstein-Uhlenbeck operator $L:=-\frac{1}{2}\Delta  + y\cdot \nabla$. 
More precisely, we consider the eigenvalue problem 
\begin{align*}
Lu&=\lambda u \text{ in } \R^{n}_+\\
u &=0\text{ on } \Lambda:=\{(y',0)\in \R^{n-1}\times \{0\}: y_{n-1}\leq 0\},\\
\p_nu&=0 \text{ on } \R^{n-1}\times\{0\}\setminus \Lambda.
\end{align*}
We further extend $u$ by even reflection in $y_n$ and obtain the eigenvalue problem in the slit domain:
\begin{align*}
Lu &=\lambda u \text{ in } \R^{n}\setminus \Lambda,\\ 
u&=0 \text{ on } \Lambda.
\end{align*}
The operator $L$ is a self-adjoint operator on $\mathcal{E}\subset L^2_{\mu}(\R^n\setminus\Lambda)$ with symmetry, where $\mathcal{E}:=\{w\in W^{1,2}_{\mu}(\R^n\setminus \Lambda): w=0 \text{ on }\Lambda, \ w \text{ is even in } y_n\}$, thus it has discrete spectrum: $0<\lambda_1\leq\lambda_2\leq \cdots$. 
It is immediate to verify that 
$$u_0(y)=\Ree(y_{n-1}+i|y_n|)^{1/2}\in \mathcal{E}$$
satisfies $Lu_0=\frac{1}{2}u_0$ and $u_0>0$ in $\R^{n}\setminus\Lambda$. Thus $u_0$ is the ground state for $L$ in $\R^{n}\setminus\Lambda$, and $\lambda_1=\frac{1}{2}$.
For the second eigenspace, we have the following characterization result:

\begin{prop}
The second eigenvalue $\lambda_2=\frac{3}{2}$. Furthermore, the corresponding eigenspace is spanned by 
\begin{align*}
\Ree(y_{n-1}+ i|y_n|)^{3/2},\quad y_i\Ree(y_{n-1}+i|y_n|)^{1/2},\quad i=1,\cdots, n-2.
\end{align*}
\end{prop}
\begin{proof}
Direct computation shows that $\Ree(y_1+i|y_2|)^{3/2}\in \mathcal{E}$ is an eigenfunction with $\lambda=3/2$. Therefore, by the characterization of the second eigenvalue:
$$\lambda_2=\inf_{\substack{w\in \mathcal{E},\\
 w \perp u_0,\ w\neq 0}}\frac{\|\nabla w\|_{L^2_\mu}^2}{\|w\|_{L^2_\mu}^2},$$
we have that $\lambda_2\leq 3/2$. The proof for $\lambda_2=3/2$ and characterization of the eigenspace is mainly based on the following property, which we are going to prove: if $u\in \mathcal{E}$ is an eigenfunction with eigenvalue $\lambda\in (1/2,5/2)$, then its tangential derivative $\p_\ell u$, $\ell=1,\cdots, n-1$, is a ground state. We also note that eigenfunctions which are $L^2$ integrable against the Gaussian measure $d\mu=e^{-|y|^2} dy$ have at most power growth at infinity (cf. for instance, \cite{CM18}).\\

\emph{(i) Case $n=2$}. 
 The proof is divided into two steps. In the first step, we will show that if $u$ is an eigenfunction with $\lambda>1/2$, then it is  in $C^{1,1/2}_{loc}(\R^2_+)$ up to the slit. This implies that its derivative $\p_1 u$ is in the space $\mathcal{E}$. In the second step we apply the characterization of the ground state to show that, if $u$ is an eigenfunction associated to $\lambda_2$, then up to a constant $u(y)=\Ree(y_1+i|y_2|)^{3/2}$.
 
\emph{Step 1: Regularity up to the slit.}
We apply the conformal transformation $z\mapsto z^2$ to open up the slit:
\begin{align*}
y\mapsto z: \quad (y_1,y_2)=(z_1^2-z_2^2,\ 2z_1z_2).
\end{align*} 
The transformation maps $\R^2\setminus\Lambda$ to $\{(z_1,z_2)\in \R^2:z_1>0\}$ and $\Lambda$ to $\{(z_1,z_2)\in \R^2:z_1=0\}$. In the new variables the original eigenvalue problem reads
\begin{align*}
-\Delta_z u = 4|z|^2 (2\lambda u-z\cdot \nabla_z u)\text{ in } \{z_1>0\},\\
u=0 \text{ on } \{z_1=0\},\quad u \text{ is even in } z_2.
\end{align*}
We can further extend $u$ to a solution in the full space $\R^2$ by odd reflection about $z_1=0$. Since the operator is hypoelliptic, solution $u$ is real-analytic, thus has the power series expansion: $u(z)=\sum_k H_k(z)$, where $H_k(z)$ is the $k$th-order homogeneous polynomial. Using the even/odd symmetry and the fact that $u$ is orthogonal to the ground state $z_1$, we have $H_0=H_1=H_2=0$. Observing that the right hand side of the equation is of order $o(|z|^2)$, then necessarily $\Delta H_3(z)=0$, which together with the symmetry gives $H_3(z)= b z_1(z_2^2-\frac{1}{3}z_1^2)$ for some constant $b$. Therefore, 
\begin{align*}
u(z)=bz_1(z_2^2-\frac{1}{3}z_1^2) + O(|z|^4).
\end{align*}
Transforming back to the original variables, we obtain
\begin{align*}
u(y)= \tilde{b} \Ree(y_1+i|y_2|)^{3/2} + O(|y|^2).
\end{align*}
In particular, this implies that $u\in C^{1,1/2}_{loc}$ up to the slit. 

\emph{Step 2: Characterization.} Now we observe that $v=\p_1u\in W^{1,2}_{\mu}(\R^2\setminus \Lambda)$ satisfies
\begin{align*}
Lv=(\lambda-1)v \text{ in } \R^2\setminus \Lambda, \quad v=0 \text{ on } \Lambda, \quad v \text{ is even in } y_2.
\end{align*}
If $\lambda\leq 3/2$, then necessarily $\lambda=3/2$ and $v(y)=\Ree(y_1+i|y_2|)^{1/2}$ (up to a constant) by the classification of the ground state. This shows the second eigenvalue $\lambda_2=3/2$. Integrating in $y_1$ leads to $u(y)=c\Ree(y_1+i|y_2|)^{3/2} + g(y_2)$ for some function $g$ depending only on $y_2$. From the equation of $u$ and the Neumann boundary condition we infer that $g$ solves the ODE
\begin{align*}
-\frac{1}{2}g'' + y_2 g' = \frac{3}{2}g \text{ in } \R, \quad g(y_2)=g(-y_2).
\end{align*}
This implies $g=0$ (one can see this from, for instance, the classical result on spectrum of the $1d$ Ornstein-Uhlenbeck operator).
Thus the eigenfunction for $\lambda_2$ is $u(y)=c\Ree(y_1+i|y_2|)^{3/2}$. \\

\emph{(ii) Case $n\geq 3$.} Let $y''=(y_1,\cdots, y_{n-2})$ and let $u$ be an eigenfunction of $L$ with eigenvalue $\lambda_2$. We observe that the boundary condition is translation invariant in $y''$. Thus, by considering the equation for the difference quotient $\Delta_h^1 u = (u(\cdot + h e'')-u(\cdot))/h$, $h\in \R$, $|h|>0$, $e''\perp \{e_{n-1}, e_n\}$, and applying the energy estimate (which is uniform in $h$), we obtain that $w=\p_iu \in \mathcal{E}$, $i=1,\cdots, n-2$. Furthermore, $w$ solves
\begin{align*}
Lw =(\lambda_2-1/2) w \text{ in }\R^n\setminus \Lambda,\quad w=0 \text{ on } \Lambda.
\end{align*}
Similarly as in (i), using $\lambda_2\leq 3/2$ and the characterization of the ground state we can conclude that $w=c\Ree(y_{n-1}+i|y_n|)^{1/2}$ and $\lambda_2=3/2$. This implies that $u(y)=\sum_{i=1}^{n-2}c_i y_i \Ree(y_{n-1}+|y_n|)^{1/2}+ h(y_{n-1}, y_n)$ for some function $h$ only depending on $y_{n-1}$ and $y_n$. From the equation and the boundary conditions for $u$, we infer that $h$ is the second Dirichlet-Neumann eigenfunction for $L$. Thus using the $2d$ result from (i) we conclude that $h=c\Ree(y_{n-1}+i|y_n|)^{3/2}$. \\

It is not hard to see that $3/2$ is the only eigenvalue in $(1/2,5/2)$. Because otherwise, from the above arguments $\lambda-1\in (1/2,3/2)$ would also be an eigenvalue, which contradicts with the spectral gap between $1/2$ and $3/2$.  The proof is now complete.
\end{proof} 

\section*{Acknowledgements}
The author would like to thank Angkana R\"uland for many helpful discussions, comments and encouragement during the preparation of the paper. The research was partially supported by CMUC -- UID/MAT/00324/2013, funded by the Portuguese Government through FCT/MCTES and co-funded by the European Regional Development Fund through the Partnership Agreement PT2020, by the FCT project TUBITAK/0005/2014, and by the DFG individual research grant SH 1403/1-1. 



\medskip
\medskip

\end{document}